\pdfoutput=1

\documentclass[11pt]{article}
\usepackage[margin=1.08in]{geometry}
\usepackage{amsmath,amssymb,amsthm,mathtools}
\usepackage{enumitem}
\usepackage{array,booktabs}
\usepackage{microtype}
\usepackage{mathrsfs}
\usepackage{tikz}
\usepackage{placeins}
\usepackage{float}
\usetikzlibrary{arrows.meta,calc,positioning,decorations.pathreplacing,fit}
\usepackage{hyperref}

\hypersetup{colorlinks=true,linkcolor=blue,citecolor=blue,urlcolor=blue}

\newtheorem{theorem}{Theorem}[section]
\newtheorem{lemma}[theorem]{Lemma}
\newtheorem{proposition}[theorem]{Proposition}
\newtheorem{corollary}[theorem]{Corollary}
\newtheorem{definition}[theorem]{Definition}
\newtheorem{remark}[theorem]{Remark}
\newtheorem{example}[theorem]{Example}

\newcommand{\Cay}{\overrightarrow{\operatorname{Cay}}}
\newcommand{\Z}{\mathbb Z}
\newcommand{\dist}{\operatorname{dist}}
\newcommand{\ord}{\operatorname{ord}}
\newcommand{\Boxprod}{\mathbin{\square}}

\tikzset{
  zpoint/.style={circle, fill, inner sep=1.6pt},
  rpoint/.style={circle, draw, thick, inner sep=1.8pt},
  negedge/.style={thick, bend left=35},
  posedge/.style={thick, dashed, bend right=35},
  core/.style={thick, dashed},
  blocker/.style={decorate, decoration={brace, amplitude=4pt}, thick},
  layerbox/.style={draw, rounded corners, thick, minimum width=1.55cm, minimum height=0.65cm, align=center, font=\scriptsize},
  pathblock/.style={draw, rounded corners, thick, minimum width=1.8cm, minimum height=0.62cm, align=center, font=\scriptsize}
}

\title{Two Arc-Disjoint Hamiltonian Paths in Finite Two-Generated Abelian Cayley Digraphs}
\author{SangHyun Park\\Yonsei University}
\date{}

\begin{document}
\maketitle

\begin{abstract}
We prove the finite abelian two-generator conjecture of Darijani--Miraftab--Witte Morris: every directed Cayley digraph on a finite abelian group with two distinct nonzero generators has two arc-disjoint Hamiltonian paths.  The proof uses a cut-reflection theorem for Hamiltonian cut values in the family $\Cay(\Z_k;a,a+1)$: if $Z$ is the set of such values and $N=k-1$, then, with $N-Z=\{N-z:z\in Z\}$,
\[
        \dist(Z,N-Z)\le1.
\]
The proof uses sector-filling inequalities for primitive-ray multiplicities and an extremal graph recording pairs at minimal reflected distance.  The estimate is sharp modulo parity: exact reflection occurs for odd $k$, while distance one occurs for even $k$.  The second remaining cyclic family, $\Cay(\Z_k;-a,a+1)$, is treated by an explicit quotient--fiber construction.  We also prove the remaining three-factor case for Cartesian products of directed cycles.  Together with the two-factor and at-least-four-factor theorems of Darijani--Miraftab--Witte Morris, this resolves their directed-cycle product conjecture for all numbers of factors.
\end{abstract}

\section{Introduction}

Hamiltonian paths and cycles in Cayley graphs are a central theme in algebraic graph theory.  They are connected to Lov\'asz's question on Hamiltonian paths in vertex-transitive graphs \cite{Lovasz1970} and to the Hamiltonian-cycle problems in Cayley graphs and digraphs surveyed in \cite{WitteGallian1984,CurranGallian1996}.  Hamilton decompositions of dense directed graphs begin with Tillson's theorem for complete directed graphs \cite{Tillson1980}, while the decomposition theory for Cayley graphs includes the survey of Alspach--Bermond--Sotteau \cite{AlspachBermondSotteau1990}, Stong's work on Cartesian products \cite{Stong1991}, and Liu's Hamilton-decomposition theorem for abelian Cayley graphs \cite{Liu2003}.  These results concern decompositions into Hamiltonian cycles in dense or undirected settings; the problem considered here is a sparse directed packing problem, asking for two arc-disjoint Hamiltonian paths in a two-generated Cayley digraph.  Directed Cayley graphs have their own difficulties: Hamiltonian cycles can fail even in small valence, and in nonabelian settings even Hamiltonian paths may fail.  Positive and negative results for Cayley digraph Hamiltonian paths include Witte \cite{Witte1982}, Jungreis \cite{Jungreis1989}, Morris \cite{Morris2012,WitteMorris2013}, and the survey of Lanel, Pallage, Ratnayake, Thevasha, and Welihinda \cite{LanelEtAl2019}.  Related work on Cayley diagrams and products of directed cycles includes Bermond--Favaron--Mah\'eo \cite{BermondFavaronMaheo1989}, Curran--Witte \cite{CurranWitte1985}, Foregger \cite{Foregger1978}, Housman \cite{Housman1981}, and Pak--Radoi\v{c}i\'c \cite{PakRadoicic2009}.  For Cartesian products of directed cycles, Trotter--Erd\H{o}s gave the two-factor Hamiltonicity characterization, and Curran--Witte proved Hamiltonicity for products of three or more nontrivial directed cycles \cite{TrotterErdos1978,CurranWitte1985}.

Darijani--Miraftab--Witte Morris (abbreviated DMM below) reduced the finite abelian two-generator problem to two cyclic assertions \cite{DMM}.  They also settled Cartesian products of directed cycles with two factors and with at least four factors, leaving the three-factor product as the remaining case.  We resolve these remaining cases: the family $\Cay(\Z_k;a,a+1)$ is settled by a reflection theorem for Hamiltonian cut values, the family $\Cay(\Z_k;-a,a+1)$ by a direct quotient--fiber construction, and the three-factor product by a switchable-pair lifting argument.  Combining the two cyclic results below with the DMM reduction theorem gives the following theorem.

\begin{theorem}[Main theorem]\label{thm:main}
Let $G$ be a finite abelian group, and let $a,b\in G\setminus\{0\}$ be distinct elements that generate $G$.  Then
\[
        \Cay(G;a,b)
\]
has two arc-disjoint Hamiltonian paths.
\end{theorem}

The two-generator assumption is essentially the minimal valence condition: for $|G|\ge3$, a one-generator directed Cayley graph has only $|G|$ non-loop arcs, fewer than the $2|G|-2$ arcs required by two disjoint Hamiltonian paths.  For $|G|=2$ the hypothesis is vacuous.

DMM formulated this finite abelian two-generator statement as Conjecture~1.3 \cite[Conjecture~1.3]{DMM}.  Thus Theorem~\ref{thm:main} gives the following immediate consequence.

\begin{corollary}\label{cor:dmm-conj-abelian}
Conjecture~1.3 of Darijani--Miraftab--Witte Morris holds for all finite abelian groups.
\end{corollary}

For finite subsets $X,Y\subseteq\Z$, we write
\[
        \dist(X,Y)=\min\{|x-y|:x\in X,\ y\in Y\}.
\]

The central structural input is a reflection theorem for Hamiltonian cut values in the first cyclic family
\[
        \Cay(\Z_k;a,a+1),
        \qquad a\not\equiv0,-1\pmod{k}.
\]
The standard cut candidates in this family are parametrized by a cut value, and the Hamiltonian ones form a finite set $Z\subseteq\{0,1,\ldots,k-1\}$.  The count criterion of DMM asks for $d,e\in Z$ with
\[
        d+e\in\{k-2,k-1,k\}.
\]
Writing $N=k-1$, Theorem~\ref{thm:near-antipodal} proves the stronger structural estimate
\[
        \dist(Z,N-Z)\le1.
\]
The DMM parametrization identifies $Z$ with a slope-ordered multiplicity sequence of primitive rays.  In this form the cut-reflection theorem may be viewed as a reflection principle for ordered primitive-ray multiplicity systems satisfying the endpoint identity and sector-filling inequalities.  We prove the reflection estimate by showing that any reflected gap of width at least two would force, through sector filling, more intervening multiplicity than the corresponding extremal configuration permits.  More concretely, if pairs at distance $\delta\ge2$ are recorded in a reflected gap graph $\Gamma_\delta$, then each fully blocked extremal edge or loop is ruled out by the lower bound
\[
        \Theta(p,q)=\#\left\{(r,s)\in\Z_{>0}^2:\frac rp+\frac sq\le1\right\}
\]
for the multiplicity between primitive rays of multiplicities $p$ and $q$.  A sharper boundary-ray estimate handles the last endpoint cap of multiplicity two.  This is a near-antipodal pair statement: it asserts the existence of a reflected pair near the center, rather than a Hausdorff-type symmetry of the whole cut set.  The estimate is sharp modulo parity: odd $k$ gives exact reflection, while even $k$ forces distance one; Example~\ref{ex:first-family} shows the distance-one case.

The second cyclic family is
\[
        \Cay(\Z_k;-a,a+1),
        \qquad 2a+1\mid k,
\]
with distinct generators.  Modulo $2a+1$, the two generators move in the same quotient direction.  A quotient-position rule with $a+2$ positions assigned to the generator $-a$ gives a Hamiltonian cycle with fiber shift $-1$; the complementary cycle cover has fiber shift $+2$ and is either Hamiltonian or splits into two cycles that can be spliced by one arc of the first cycle.

We also prove the remaining three-factor case in the Cartesian-product problem for directed cycles.  Darijani, Miraftab, and Witte Morris proved the two-factor case and the case of at least four factors, and treated several three-factor subcases.  Theorem~\ref{thm:threecycles} proves the remaining three-factor case; combined with their two-factor and at-least-four-factor results, it gives Corollary~\ref{cor:all-cycle-products}.  The proof isolates a switchability condition that lifts a pair of disjoint Hamiltonian paths through one directed-cycle factor.

\begin{theorem}\label{thm:threecycles}
For all $m,n,\ell\ge2$, the Cartesian product
\[
        \vec C_m\Boxprod \vec C_n\Boxprod \vec C_\ell
\]
has two arc-disjoint Hamiltonian paths.
\end{theorem}

\begin{corollary}\label{cor:all-cycle-products}
For every $r\ge2$, the Cartesian product of $r$ directed cycles of length at least two has two arc-disjoint Hamiltonian paths.  In particular, Conjecture~1.1 of Darijani--Miraftab--Witte Morris holds.
\end{corollary}

The proof of Corollary~\ref{cor:all-cycle-products} is given at the end of Section~\ref{sec:threecycles}, after Theorem~\ref{thm:threecycles} is proved.  Further directions are discussed after the proofs.

\section{Preliminaries from Darijani--Miraftab--Witte Morris}

All groups are written additively.  The directed Cayley digraph $\Cay(G;a,b)$ has vertex set $G$ and arcs
\[
        x\to x+a,
        \qquad
        x\to x+b
        \qquad (x\in G).
\]
If $P$ is a path or a spanning path-cycle cover, let $\delta_b(P)$ denote the number of arcs of $P$ labelled by $b$.

We use several results of Darijani, Miraftab, and Witte Morris \cite{DMM}.  The first is their count criterion for producing disjoint Hamiltonian paths from two Hamiltonian paths with suitable arc counts.  The second is the index-one cyclic specialization of their lattice parametrization of Hamiltonian cut values in the family with \(a-b=-1\), written below in the notation used in Section~\ref{sec:first-family}.  The third is their reduction of the finite abelian two-generator problem to two cyclic assertions.  The final two results are used for the directed-cycle product theorem in Section~\ref{sec:threecycles}.

\begin{theorem}[Count criterion, Darijani--Miraftab--Witte Morris, Proposition~4.3]\label{input:count}
Let $G$ be a finite abelian group generated by $a,b$.  Suppose there are Hamiltonian paths $P,Q$ in $\Cay(G;a,b)$, not necessarily distinct, such that
\[
        \delta_b(P)+\delta_b(Q)
        \in
        \{|G|-2,|G|-1,|G|\}.
\]
Then $\Cay(G;a,b)$ has two arc-disjoint Hamiltonian paths.
\end{theorem}

\begin{theorem}[Lattice parametrization, Darijani--Miraftab--Witte Morris]\label{input:lattice}
In $\Cay(\Z_k;a,a+1)$, let
\[
        m=\ord(a),
        \qquad
        n=|\Z_k:\langle a\rangle|,
        \qquad
        k=mn,
\]
Since $n=|\Z_k:\langle a\rangle|$, the element $n(a+1)$ belongs to the cyclic subgroup $\langle a\rangle$.  Let $e$ be determined by
\[
        n(a+1)\equiv ea \pmod{k},
        \qquad 0\le e<m.
\]
Let
\[
        T=\operatorname{conv}\{(0,0),(n,0),(e,m)\},
        \qquad
        L(x,y)=mx+(n-e)y,
\]
and put $N=k-1=mn-1$.  In this family $a-(a+1)=-1$, so the arc-forcing index of Darijani--Miraftab--Witte Morris is one and the single triangle $T_0$ is the triangle $T$ above.  Let $A_1,\ldots,A_f$ be the primitive ray generators obtained by taking the two boundary primitive directions of $T$ and every non-boundary primitive ray generator $R$ for which at least one positive multiple of $R$ satisfies $L\le N$, ordered by slope.  If
\[
        H_r=\left\lfloor\frac{N}{L(A_r)}\right\rfloor,
\]
then the Hamiltonian cut values are
\[
        U_r=H_1+2\sum_{q=2}^{r}H_q,
        \qquad 1\le r<f.
\]
Thus
\[
        Z=\{U_1,\ldots,U_{f-1}\}.
\]
Moreover
\[
        U_{f-1}+H_f=N.
\]
\end{theorem}

This theorem is the $a-b=-1$ specialization of the notation and results of DMM \cite[Notation~3.19, Theorem~3.21, Lemma~3.23]{DMM}.  The arc-forcing index is one, so their only triangle is $T$, and the ordered list above is the ray list used in Section~\ref{sec:first-family}.  For every ray generator $R$ in this list,
\[
        H(R)=\left\lfloor \frac{N}{L(R)}\right\rfloor
\]
is the number of positive multiples
\[
        R,2R,\ldots,H(R)R
\]
on that ray with $L\le N$.  Each boundary multiplicity may be zero, and both may be zero in some cases.  Sector estimates below are applied only to blocks of positive multiplicity.

In the nondegenerate family considered in Theorem~\ref{thm:first-family}, the triangle has two boundary directions, so $f\ge2$ and $Z$ is nonempty.  Throughout Section~\ref{sec:first-family}, a sum of multiplicities between two rays is a sum of the counted multiples $H(R)$ over the intervening ray generators in this ray list.

\begin{example}[The lattice encoding for $k=10$, $a=4$]\label{ex:lattice-10-4}
Here $m=\ord_{\Z_{10}}(4)=5$, $n=\gcd(10,4)=2$, and $e=0$, so
\[
        T=\operatorname{conv}\{(0,0),(2,0),(0,5)\},
        \qquad L(x,y)=5x+2y,
        \qquad N=9.
\]
The primitive ray generators in slope order are
\[
        (1,0),\ (1,1),\ (1,2),\ (0,1),
\]
with multiplicities
\[
        1,\quad 1,\quad 1,\quad 4.
\]
Thus
\[
        U_1=1,
        \qquad U_2=3,
        \qquad U_3=5,
\]
and hence $Z=\{1,3,5\}$.  Here $c_L=1$, $c_R=N-5=4$, and the two internal gap multiplicities are both $1$.  This example will be used again after the cut-reflection theorem.
\end{example}

\begin{theorem}[Reduction to two cyclic families, Darijani--Miraftab--Witte Morris]\label{input:reduction}
Assume the following two cyclic assertions hold.
\begin{enumerate}[label=(\roman*)]
\item For every $k\ge3$ and every $a\in\Z_k$ with $a\not\equiv0,-1\pmod{k}$, the digraph $\Cay(\Z_k;a,a+1)$ has two arc-disjoint Hamiltonian paths.
\item For every $a\ge1$ and $k=(2a+1)L$ with $L\ge2$, the digraph $\Cay(\Z_k;-a,a+1)$ has two arc-disjoint Hamiltonian paths.
\end{enumerate}
Then Theorem~\ref{thm:main} holds.
\end{theorem}

Sections~\ref{sec:first-family} and~\ref{sec:second-family} prove the two assertions required by the reduction.  In the second assertion, $L=1$ would give duplicate generators, so it is outside the two-element generating-set hypothesis.

\begin{theorem}[Two directed cycles, Darijani--Miraftab--Witte Morris~{\cite[Theorem~4.4]{DMM}}]\label{input:twocycle-existence}
For all $m,n\ge2$, the Cartesian product $\vec C_m\Boxprod\vec C_n$ has two arc-disjoint Hamiltonian paths.
\end{theorem}

\begin{lemma}[Endpoint-and-coset consequence of Darijani--Miraftab--Witte Morris]\label{input:twocycle-structure}
Let
\[
        D=\Cay(\Z_m\times\Z_n;a,b)
\]
be a product of two directed cycles, and assume that $D$ has no Hamiltonian directed cycle.  Then there exist arc-disjoint Hamiltonian paths $P,Q$ in $D$, with initial and terminal vertices
\[
        \iota_P,\tau_P,
        \qquad
        \iota_Q,\tau_Q,
\]
satisfying the following three properties simultaneously.
\begin{enumerate}[label=(\roman*),leftmargin=*]
\item \emph{Endpoint rigidity} \cite[Proposition~4.6]{DMM}.  One has
\[
        \iota_Q\in\{\tau_P+a,\tau_P+b\},
        \qquad
        \tau_Q\in\{\iota_P-a,\iota_P-b\}.
\]
The assumption that $D$ has no Hamiltonian directed cycle ensures that no generator edge joins $\tau_P$ to $\iota_P$; otherwise $P$ would close to such a cycle.

\item \emph{Terminal-coset interval and endpoint formula} \cite[Lemma~3.9]{DMM}.  Put $w=a-b$, write $L=\ord(w)$, and write the terminal coset of $P$ as
\[
        x_r=\tau_P+rw,
        \qquad r\in\Z_L.
\]
There is an integer $d$, $0\le d<L$, such that
\[
        P_b=\{x_1,\ldots,x_d\},
        \qquad
        P_a=\{x_{d+1},\ldots,x_{L-1}\}.
\]
Here $P_b$ and $P_a$ are the sets of vertices in the terminal coset whose outgoing edge in $P$ is labelled by $b$ and by $a$, respectively.  The same description determines the initial endpoint:
\[
        \iota_P=\tau_P+a+dw.
\]
The analogous formula applies to any path to which the same terminal-coset description is applied.

\item \emph{Uniform travel away from the terminal coset} \cite[Lemma~3.9]{DMM}.  The path $Q$ has the same terminal coset as $P$.  On every other coset of $\langle w\rangle$, both paths travel uniformly, and they use opposite labels on each such coset.
\end{enumerate}
\end{lemma}

\begin{proof}
Choose $P,Q$ as in \cite[Theorem~4.4]{DMM}.  These paths are fixed throughout the proof; the endpoint rigidity and terminal-coset description cited below are applied to this same pair, and no new choice of paths is made.  Since $D$ has no Hamiltonian directed cycle, no generator arc joins $\tau_P$ to $\iota_P$; otherwise $P$ would close to such a cycle.  \cite[Proposition~4.6]{DMM} therefore gives the endpoint rigidity in (i).

Put $w=a-b$.  By \cite[Lemma~3.9]{DMM}, the terminal-coset interval for $P$ gives, in particular, the endpoint formula
\[
        \iota_P=\tau_P+a+dw
\]
for some $d$.  Hence
\[
        \iota_P-a=\tau_P+dw,
        \qquad
        \iota_P-b=\tau_P+(d+1)w.
\]
The second endpoint relation in (i) gives
\[
        \tau_Q\in\{\iota_P-a,\iota_P-b\}
        \subseteq \tau_P+\langle w\rangle.
\]
Thus $P$ and $Q$ have the same terminal coset.  Applying the same lemma to both paths gives uniform travel on every nonterminal coset.  If the two paths used the same label on such a coset, they would contain the same outgoing Cayley arc from every vertex of that coset, contradicting arc-disjointness.  Hence the labels on each nonterminal coset are opposite.
\end{proof}

\section{The first cyclic family}\label{sec:first-family}

We prove the first cyclic assertion in Theorem~\ref{input:reduction}.

\begin{theorem}\label{thm:first-family}
Let $k\ge3$, and let $a\in\Z_k$ satisfy $a\not\equiv0,-1\pmod{k}$.  Then
\[
        D(k,a)=\Cay(\Z_k;a,a+1)
\]
has two arc-disjoint Hamiltonian paths.
\end{theorem}

Throughout this section put
\[
        b=a+1,
        \qquad
        N=k-1.
\]
Since $a-b=-1$, the arc-forcing subgroup is all of $\Z_k$.

\subsection{Cut permutations and the Hamiltonian cut set}

When defining cut candidates, we identify $\Z_k$ with the ordered set of representatives $\{0,1,\ldots,k-1\}$.  For $0\le d<k$, define the standard cut candidate $P_d$ by prescribing the outgoing arc at every vertex except $d$:
\[
P_d(x)=
\begin{cases}
        x+b, & 0\le x<d,\\
        \text{undefined}, & x=d,\\
        x+a, & d<x\le k-1.
\end{cases}
\]
Thus $P_d$ uses $d$ arcs labelled by $b$.

Adjoin the formal closing successor $d\mapsto a$, and denote the resulting permutation of $\Z_k$ by
\[
\Phi_d(i)=
\begin{cases}
        i+a+1, & 0\le i<d,\\
        a, & i=d,\\
        i+a, & d<i\le k-1.
\end{cases}
\]
This formal closing successor turns the cut candidate $P_d$ into the permutation $\Phi_d$.  If $\Phi_d$ is a $k$-cycle, deleting the formal closing successor $d\mapsto a$ leaves a directed Hamiltonian path from $a$ to $d$ in $D(k,a)$.  Conversely, if $P_d$ is a Hamiltonian path, then adjoining the formal closing successor gives a $k$-cycle.  Hence
\[
        P_d\text{ is Hamiltonian}
        \quad\Longleftrightarrow\quad
        \Phi_d\text{ is a }k\text{-cycle}.
\]

Define the Hamiltonian cut set
\[
        Z=\{d:\Phi_d\text{ is a }k\text{-cycle}\}.
\]
Thus every $d\in Z$ is realized by the Hamiltonian cut path obtained from $P_d$, and $\delta_b(P_d)=d$.  By the lattice parametrization, Theorem~\ref{input:lattice},
\[
        Z=\{U_1,\ldots,U_{f-1}\}.
\]

We shall use the lattice notation of Theorem~\ref{input:lattice}.  Thus $A_1,\ldots,A_f$ are the primitive ray generators in $T$ ordered by slope, with multiplicities
\[
        H_r=\left\lfloor\frac{N}{L(A_r)}\right\rfloor,
\]
and
\[
        Z=\{U_1,\ldots,U_{f-1}\},
        \qquad
        U_r=H_1+2\sum_{q=2}^{r}H_q.
\]
When $Z$ is written in increasing order as
\[
        Z=\{z_0<z_1<\cdots<z_s\},
\]
we write
\[
        z_{r+1}-z_r=2\lambda_r
        \qquad(0\le r<s),
\]
and we set
\[
        c_L=z_0,
        \qquad
        c_R=N-z_s.
\]

\begin{lemma}[DMM-to-gap dictionary]\label{lem:gap-dictionary}
In the notation above, $s=f-2$ and
\[
        z_r=U_{r+1}\quad(0\le r\le s),
        \qquad
        c_L=H_1,
        \qquad
        \lambda_r=H_{r+2}\quad(0\le r<s),
        \qquad
        c_R=H_f.
\]
Thus each internal gap $[z_r,z_{r+1}]$ corresponds to the primitive ray generator $A_{r+2}$, with the same multiplicity, and the order of internal gaps is the slope order of the corresponding internal primitive rays.  In particular all cut values have the same parity, namely the parity of $c_L$.
\end{lemma}

\begin{proof}
By Theorem~\ref{input:lattice}, the cut values are $U_1,\ldots,U_{f-1}$ in increasing slope order, so $s=f-2$ and $z_r=U_{r+1}$.  Hence $c_L=z_0=U_1=H_1$.  For $0\le r<s$,
\[
        z_{r+1}-z_r
        =U_{r+2}-U_{r+1}
        =2H_{r+2},
\]
so $\lambda_r=H_{r+2}$.  Finally, Theorem~\ref{input:lattice} gives $U_{f-1}+H_f=N$, hence $c_R=N-z_s=H_f$.  Since the ray list is ordered by slope, the same order is inherited by the internal gaps.
\end{proof}

For reference, we summarize the notation used in the cut-reflection proof.
\[
\begin{array}{c|c|c}
\text{object} & \text{notation} & \text{role in Section~\ref{sec:first-family}}\\
\hline
\text{primitive ray} & A_r & \text{slope-ordered ray in }T\\
\text{multiplicity} & H_r=\lfloor N/L(A_r)\rfloor & \text{allowed multiples on }A_r\\
\text{cut value} & U_r & \text{Hamiltonian }b\text{-arc count}\\
\text{cut set} & Z=\{z_0<\cdots<z_s\} & \text{ordered Hamiltonian cut values}\\
\text{internal gap} & z_{r+1}-z_r=2\lambda_r & \lambda_r=H_{r+2}\\
\text{endpoint caps} & c_L=z_0,\ c_R=N-z_s & \text{boundary multiplicities}
\end{array}
\]

\begin{lemma}[Boundary multiplicities]\label{lem:endpoint-caps}
Let
\[
        c_L=z_0,
        \qquad
        c_R=N-z_s.
\]
Then
\[
        c_L=\gcd(k,a)-1,
        \qquad
        c_R=\gcd(k,a+1)-1.
\]
\end{lemma}

\begin{proof}
By Theorem~\ref{input:lattice}, $z_0=U_1=H_1$.  The first primitive ray is $(1,0)$, so $L(1,0)=m$.  Hence
\[
        H_1=\left\lfloor\frac{mn-1}{m}\right\rfloor=n-1.
\]
Since $n=|\Z_k:\langle a\rangle|=\gcd(k,a)$, this gives $c_L=\gcd(k,a)-1$.

For the right cap, Theorem~\ref{input:lattice} gives $c_R=N-U_{f-1}=H_f$.  The last boundary point is $(e,m)$.  Let $d_0=\gcd(e,m)$; the corresponding primitive ray is $(e/d_0,m/d_0)$.  Therefore
\[
        H_f=\left\lfloor \frac{mn-1}{mn/d_0}\right\rfloor=d_0-1.
\]
We claim that $d_0=\gcd(k,a+1)$.  Since $n(a+1)\equiv ea\pmod{k}$, the orders of $n(a+1)$ and $ea$ in $\Z_k$ are equal.  The order of $ea$ in the cyclic subgroup $\langle a\rangle$ is $m/d_0$.  Let $h=\gcd(k,a+1)$.  Since $n=\gcd(k,a)$ divides $a$, we have $\gcd(n,a+1)=1$.  Hence $h=\gcd(m,a+1)$ and
\[
        \gcd(k,n(a+1))=n\gcd(m,a+1)=nh.
\]
Thus the order of $n(a+1)$ in $\Z_k$ is $mn/(nh)=m/h$.  Hence $d_0=h$, and so
\[
        c_R=H_f=h-1=\gcd(k,a+1)-1.
\]
\end{proof}

We shall regard the two endpoint caps and the internal gaps as blocks.  Let
\[
        \mathcal B_L,
        G_0,G_1,\ldots,G_{s-1},
        \mathcal B_R
\]
be the ordered block sequence, where $G_i$ is the internal gap $[z_i,z_{i+1}]$.  Define the height of a block by
\[
        h(\mathcal B_L)=c_L,
        \qquad
        h(G_i)=\lambda_i,
        \qquad
        h(\mathcal B_R)=c_R.
\]
For two blocks $P<Q$ in this order, with positive heights when sector filling is applied, let
\[
        M(P,Q)=\sum_{P<B<Q}h(B)
\]
be the multiplicity mass strictly between them.  Thus, for instance,
\[
        M(\mathcal B_L,G_i)=\sum_{r=0}^{i-1}\lambda_r=\frac{z_i-c_L}{2},
        \qquad
        M(G_t,G_j)=\sum_{r=t+1}^{j-1}\lambda_r.
\]
By Lemma~\ref{lem:gap-dictionary}, this is the same multiplicity mass as the corresponding sum over primitive ray generators in slope order.  The notation below keeps the block, such as $G_t$, distinct from its height $h(G_t)=\lambda_t$.

Thus the cut set is governed by a block sequence rather than by unrelated cut values: the endpoint heights are the gcd caps of Lemma~\ref{lem:endpoint-caps}, the internal heights are the DMM primitive-ray multiplicities, and $Z$ is the corresponding prefix-sum set.  The sector estimates below express the resulting separation constraint for large blocks: two large heights force substantial multiplicity mass between them.

\subsection{Sector filling}\label{sec:first-family-sector}

Let $R,S$ be two primitive ray generators of the triangle $T$, with $R<S$ in slope order.  Write
\[
        p=H(R),
        \qquad
        q=H(S).
\]
Let
\[
        M(R,S)=\sum_{R<C<S}H(C),
\]
where the sum runs over primitive rays strictly between $R$ and $S$.  We call such sums total multiplicities between the indicated rays.

Figure~\ref{fig:sector-filling} illustrates the injection used in the next lemma.

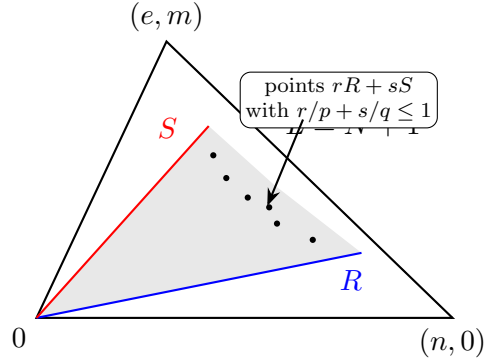
\begin{figure}[t]
\centering
\begin{tikzpicture}[scale=0.86, >=Stealth]
\coordinate (O) at (0,0);
\coordinate (X) at (6.4,0);
\coordinate (Y) at (2.0,4.25);
\draw[thick] (O)--(X)--(Y)--cycle;
\node[below left] at (O) {$0$};
\node[below] at (X) {$(n,0)$};
\node[above] at (Y) {$(e,m)$};
\coordinate (P) at (5.0,1.0);
\coordinate (Q) at (2.65,2.95);
\fill[gray!18] (O)--(P)--(3.65,2.05)--(Q)--cycle;
\draw[thick, blue] (O)--(P) node[pos=.90, below right] {$R$};
\draw[thick, red] (O)--(Q) node[pos=.88, above left] {$S$};
\draw[dashed] (X)--(Y) node[pos=.62, above right] {$L=N+1$};
\foreach \x/\y in {3.7/1.45,3.25/1.85,2.92/2.15,4.25/1.20,3.58/1.70,2.72/2.50}{
  \fill[black] (\x,\y) circle (1.35pt);
}
\node[draw, rounded corners, fill=white, align=center, font=\scriptsize, inner sep=2pt] at (4.65,3.35) {points $rR+sS$\\with $r/p+s/q\le1$};
\draw[->, thick] (4.10,3.05)--(3.55,1.72);
\end{tikzpicture}
\caption{The sector-filling injection.  Points $rR+sS$ with $r/p+s/q\le1$ lie in the open sector between $R$ and $S$ and below the boundary of $T$.}
\label{fig:sector-filling}
\end{figure}

\begin{lemma}[Sector filling]\label{lem:sector-filling}
With the notation above,
\[
        M(R,S)
        \ge
        \Theta(p,q),
\]
where
\[
        \Theta(p,q)=
        \#\left\{(r,s)\in\Z_{>0}^2:
        \frac rp+\frac sq\le1\right\}.
\]
\end{lemma}

\begin{proof}
Since $H(R)=p$ and $H(S)=q$, we have
\[
        L(R)\le \frac Np,
        \qquad
        L(S)\le \frac Nq.
\]
For $r,s>0$, the lattice point
\[
        V_{r,s}=rR+sS
\]
lies in the open sector between $R$ and $S$.  If $r/p+s/q\le1$, then
\[
        L(V_{r,s})=rL(R)+sL(S)
        \le
        N\left(\frac rp+\frac sq\right)
        \le N.
\]
Since the opposite side of $T$ is the line $L=mn=N+1$, the inequality $L(V_{r,s})\le N$ places $V_{r,s}$ strictly inside $T$.  Write $V_{r,s}=tC$, where $C$ is primitive.  Then $C$ lies between $R$ and $S$, and $tL(C)\le N$, so this point is counted among the $H(C)$ allowed multiples on the ray $C$.

The representation $rR+sS$ is unique, because $R$ and $S$ are linearly independent.  Thus distinct pairs $(r,s)$ give distinct lattice points.  If several pairs map to multiples of the same primitive ray $C$, their images are distinct multiples $tC$; since $H(C)$ counts these allowed multiples with multiplicity, injectivity is preserved at the multiset level.  Hence the admissible pairs inject into the multiset counted by $M(R,S)$.
\end{proof}

\begin{lemma}\label{lem:theta-bound}
For positive integers $p,q$,
\[
        \Theta(p,q)\ge \frac{(p-1)(q-1)}2.
\]
\end{lemma}

\begin{proof}
We have
\[
\Theta(p,q)
=
\sum_{r=1}^{p-1}
\left\lfloor q\left(1-\frac rp\right)\right\rfloor
=
\sum_{r=1}^{p-1}
\left\lfloor \frac{qr}{p}\right\rfloor.
\]
The second equality is obtained by replacing the summation index $r$ by $p-r$.  Pairing the terms for $r$ and $p-r$ gives
\[
\left\lfloor\frac{qr}{p}\right\rfloor+
\left\lfloor\frac{q(p-r)}{p}\right\rfloor
\ge q-1.
\]
If $p$ is odd, the $p-1$ terms form $(p-1)/2$ such pairs, and the desired bound follows.  If $p$ is even, the $p-2$ terms away from $r=p/2$ form $(p-2)/2$ such pairs, and the unpaired middle term is
\[
        \left\lfloor\frac q2\right\rfloor\ge\frac{q-1}{2}.
\]
Hence
\[
        \Theta(p,q)
        \ge
        \frac{p-2}{2}(q-1)+\frac{q-1}{2}
        =
        \frac{(p-1)(q-1)}2.
\]
\end{proof}

\begin{lemma}\label{lem:theta-linear}
If $q\ge p\ge2$, then $\Theta(p,q)\ge p-1$.
\end{lemma}

\begin{proof}
For each $1\le r\le p-1$, the pair $(r,1)$ is admissible, because
\[
        \frac rp+\frac1q
        \le
        \frac{p-1}{p}+\frac1p
        =1.
\]
Thus $\Theta(p,q)\ge p-1$.
\end{proof}

\begin{lemma}[Comparison estimate]\label{lem:absorption}
Let $\delta\ge2$.  If
\[
        p>\delta,
        \qquad
        q\ge q_0>\delta,
\]
then
\[
        \Theta(p,q)>q_0-\delta.
\]
\end{lemma}

\begin{proof}
By Lemma~\ref{lem:theta-bound},
\[
        \Theta(p,q)\ge\frac{(p-1)(q-1)}2
        \ge
        \frac{\delta(q_0-1)}2.
\]
It remains only to check that
\[
        \frac{\delta(q_0-1)}2>q_0-\delta
\]
is equivalent to
\[
        q_0(\delta-2)+\delta>0.
\]
\end{proof}

\begin{lemma}[Endpoint sector filling]\label{lem:endpoint-sector}
Let $E$ be an endpoint primitive ray with endpoint multiplicity $c=H(E)>0$.  If $E$ is the left endpoint ray, let $R$ be an internal primitive ray to its right in the slope order; if $E$ is the right endpoint ray, let $R$ be an internal primitive ray to its left.  Put $h=H(R)$, and let $M(E,R)$ denote the total multiplicity of the primitive rays strictly between $E$ and $R$.  Then
\[
        M(E,R)\ge \Theta(c,h).
\]
The same assertion holds at the right endpoint after reversing the order of the primitive rays.
\end{lemma}

\begin{proof}
The proof is the sector-filling injection of Lemma~\ref{lem:sector-filling} with one of the two rays placed on the boundary of the primitive-ray order.  For every pair of positive integers $(r,s)$ with
\[
        \frac r c+\frac s h\le1,
\]
the lattice point $rE+sR$ lies in the open sector between $E$ and $R$, and
\[
        L(rE+sR)\le r\frac Nc+s\frac Nh\le N<mn.
\]
Hence the primitive ray through $rE+sR$ is an admissible ray of the triangle between $E$ and $R$.  If several such points lie on the same primitive ray, they contribute as distinct allowed multiples, as in Lemma~\ref{lem:sector-filling}.  This gives an injection into the multiset counted by $M(E,R)$, and the reflected endpoint is identical after reversing the order.
\end{proof}

\begin{corollary}[No adjacent large multiplicities]\label{cor:no-adjacent-large}
Adjacent primitive rays cannot both have multiplicity at least $2$.  The same statement holds for an endpoint cap and the adjacent internal multiplicity.
\end{corollary}

\begin{proof}
If two adjacent blocks have multiplicities $p,q\ge2$, then the intervening total multiplicity is $0$ but Lemmas~\ref{lem:sector-filling} and \ref{lem:theta-linear} give it at least $\Theta(p,q)\ge1$.  The endpoint case follows from Lemma~\ref{lem:endpoint-sector}.
\end{proof}

\subsection{The reflected gap graph}

The rest of the proof of the cut-reflection theorem is by contradiction.  We assume that the least distance between $Z$ and its reflection is $\delta\ge2$, record the pairs realizing this distance in a graph $\Gamma_\delta$, and reduce the graph to a fully blocked edge or an isolated fully blocked loop.  The reduction leaves four canonical alternatives: an internal edge, an internal loop, a one-sided endpoint-blocked edge with internal reflected point, or an endpoint-core placement.  The common calculation is Lemma~\ref{lem:reflected-defect}: after passing to an adjacent cut value across an internal blocking gap, reflection gives an upper bound for a block mass, while sector filling gives a lower bound for the same mass.
Assume, for contradiction, that
\[
        \delta=\dist(Z,N-Z)\ge2.
\]
Then $|u+v-N|\ge\delta$ for all $u,v\in Z$, and equality occurs for at least one pair because $Z$ is finite and nonempty.

\begin{definition}[Reflected gap graph]
The graph $\Gamma_\delta$ has vertex set $Z$.  A negative edge joins $u,v\in Z$ if
\[
        u+v=N-\delta,
\]
and a positive edge joins $u,v\in Z$ if
\[
        u+v=N+\delta.
\]
Loops are allowed.

\end{definition}

The graph is nonempty: since $Z$ is finite and nonempty, some pair $u,v\in Z$ realizes $|u+v-N|=\delta$, and this pair is a negative or positive edge of $\Gamma_\delta$, with a loop allowed when $u=v$.  Crossing an internal gap of multiplicity $\delta$ changes the sign of an edge, as recorded in Lemma~\ref{lem:propagation} and illustrated in Figure~\ref{fig:gap-propagation}.

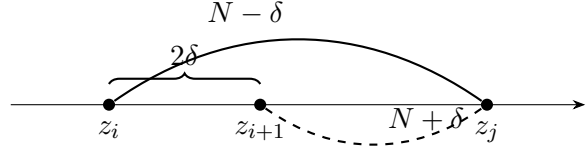
\begin{figure}[H]
\centering
\begin{tikzpicture}[x=1.0cm,y=0.65cm,>=Stealth]
\draw[->] (-0.3,0)--(7.3,0);
\node[zpoint,label=below:$z_i$] (zi) at (1,0) {};
\node[zpoint,label=below:$z_{i+1}$] (zip) at (3,0) {};
\node[zpoint,label=below:$z_j$] (zj) at (6,0) {};
\draw[blocker] (1,0.42)--(3,0.42) node[midway,above=3pt] {$2\delta$};
\draw[negedge] (zi) to node[above left=2pt] {$N-\delta$} (zj);
\draw[posedge] (zip) to node[above right=2pt] {$N+\delta$} (zj);
\end{tikzpicture}
\caption{Propagation across a $\delta$-gap: if $z_i+z_j=N-\delta$ and $z_{i+1}-z_i=2\delta$, then $z_{i+1}+z_j=N+\delta$.  When the equality $\lambda_i=\delta$ fails, the extension is blocked; these blocked sides are the cases considered in Lemma~\ref{lem:sealed-exhaustion}.}
\label{fig:gap-propagation}
\end{figure}

\begin{definition}[$\delta$-cores and blocking conditions]\label{def:blocked}
An internal gap $[z_t,z_{t+1}]$ has $\delta$-core, or $\delta$-interior,
\[
        [z_t+\delta,\ z_{t+1}-\delta],
\]
which is nonempty when $\lambda_t\ge\delta$.
The left endpoint cap has $\delta$-core $[0,c_L-\delta]$ if $c_L\ge\delta$, and the right endpoint cap has $\delta$-core $[N-c_R+\delta,N]$ if $c_R\ge\delta$.

For a negative edge $z_i z_j$, $z_i+z_j=N-\delta$, the possible propagation at $z_i$ is to the right.  It is internally blocked at $z_i$ if $i<s$ and $\lambda_i>\delta$; it is endpoint-blocked at $z_i$ if $i=s$.  We say that the edge is fully blocked if the right propagation is blocked at both endpoints, and internally blocked on both sides if both blocking gaps are internal gaps.

For a positive edge $z_i z_j$, $z_i+z_j=N+\delta$, the possible propagation at $z_i$ is to the left.  It is internally blocked at $z_i$ if $i>0$ and $\lambda_{i-1}>\delta$; it is endpoint-blocked at $z_i$ if $i=0$.  A positive edge is fully blocked, respectively internally blocked on both sides, by the analogous condition at both endpoints.
\end{definition}

\begin{lemma}[Location of reflected points]\label{lem:target-placement}
Assume $\delta=\dist(Z,N-Z)\ge2$.  Let $x\in Z$ and put $T=N-x$.  Then $T$ lies either in the $\delta$-core of an internal gap of $Z$, or in the $\delta$-core of one of the two endpoint caps.
\end{lemma}

\begin{proof}
For every $z\in Z$ we have
\[
        |T-z|=|N-x-z|\ge\delta.
\]
Since $0\le T\le N$, there are three possibilities.  If $z_t\le T\le z_{t+1}$ for some internal gap, then
\[
        T-z_t\ge\delta,
        \qquad
        z_{t+1}-T\ge\delta,
\]
so $T\in[z_t+\delta,z_{t+1}-\delta]$.  If $T<z_0=c_L$, then $z_0-T\ge\delta$, so $T\le c_L-\delta$ and $T$ lies in the left endpoint core.  If $T>z_s=N-c_R$, then $T-z_s\ge\delta$, so $T\ge N-c_R+\delta$ and $T$ lies in the right endpoint core.
\end{proof}

\begin{lemma}[Reflection]\label{lem:reflection}
Let $\widetilde Z=N-Z$, written in increasing order.  Then
\[
        \dist(Z,N-Z)=\dist(\widetilde Z,N-\widetilde Z).
\]
The involution $z\mapsto N-z$ interchanges negative and positive edges, interchanges the left and right boundary multiplicities, and reverses the internal gap sequence.  More explicitly, if
\[
        \widetilde Z=N-Z=\{\widetilde z_0<\cdots<\widetilde z_s\},
\]
then
\[
        \widetilde z_r=N-z_{s-r},
        \qquad
        \widetilde\lambda_r=\lambda_{s-1-r},
        \qquad
        \widetilde c_L=c_R,
        \qquad
        \widetilde c_R=c_L.
\]
Consequently internal blocking, endpoint blocking, and the corresponding endpoint-core alternatives are preserved under reflection, with left and right interchanged.
\end{lemma}

\begin{proof}
The sign change follows from
\[
        (N-u)+(N-v)-N=-(u+v-N).
\]
The displayed formulas follow from the increasing ordering of \(N-Z\).
\end{proof}

\begin{lemma}[Propagation]\label{lem:propagation}
If $z_i+z_j=N-\delta$ and $i<s$, then $\lambda_i\ge\delta$.  If $\lambda_i=\delta$, then
\[
        z_{i+1}+z_j=N+\delta.
\]
Dually, if $z_i+z_j=N+\delta$ and $i>0$, then $\lambda_{i-1}\ge\delta$, and equality gives
\[
        z_{i-1}+z_j=N-\delta.
\]
\end{lemma}

\begin{proof}
For the negative edge,
\[
        z_{i+1}+z_j-N=-\delta+2\lambda_i.
\]
The definition of $\delta$ gives $|-\delta+2\lambda_i|\ge\delta$, so $\lambda_i\ge\delta$.  Equality gives the displayed positive edge.  The positive case is identical after replacing $z_{i+1}$ by $z_{i-1}$.
\end{proof}

\begin{lemma}\label{lem:isolated-loop}
A vertex with a loop in $\Gamma_\delta$ has no non-loop incident edge.
\end{lemma}

\begin{proof}
Consider a negative loop $2z_i=N-\delta$.  If there is also a positive edge at $z_i$, then $z_i+2\delta\in Z$.  No point of $Z$ lies strictly between $z_i$ and $z_i+2\delta$, since such a point would give a pair closer than $\delta$ to $N$.  Thus the right block at $z_i$ has multiplicity $\delta$.

On the other side, if $z_i-2\delta\in Z$, then no point of $Z$ lies strictly between $z_i-2\delta$ and $z_i$, again by extremality of $\delta$, so the left block has multiplicity $\delta$.  If $z_i-2\delta\notin Z$ and $i>0$, write $z_i-z_{i-1}=2\lambda$.  If $\lambda<\delta$, then
\[
        (z_i+2\delta)+z_{i-1}-N
        =\delta-2\lambda,
\]
whose absolute value is less than $\delta$, contradicting the definition of $\delta$.  Hence the left internal block has multiplicity at least $\delta$.  If $i=0$, then the left cap is $c_L=z_i$; because the incident positive edge gives $z_i+2\delta\in Z\subseteq[0,N]$ and $N=2z_i+\delta$, we have $z_i\ge\delta$.  Thus the missing extension is blocked by a left block or endpoint cap of multiplicity at least $\delta$.  In all cases two adjacent blocks have multiplicity at least $2$, contradicting Corollary~\ref{cor:no-adjacent-large}.  Positive loops follow by reflection.
\end{proof}

\begin{lemma}[Non-loop matching]\label{lem:degree-path}
In the subgraph of $\Gamma_\delta$ consisting only of non-loop edges, every component that contains an edge is a sign-alternating path.  No such path has length greater than one.
\end{lemma}

\begin{proof}
A vertex $u$ has at most one negative non-loop neighbor, namely $N-\delta-u$, and at most one positive non-loop neighbor, namely $N+\delta-u$.  Thus the degree in the non-loop subgraph is at most two, and degree two forces one edge of each sign.

There is no non-loop cycle.  To see this explicitly, let an alternating walk start with a negative edge from $u_0$ to $v_0$, so $u_0+v_0=N-\delta$.  If the next edge at $v_0$ is positive, its other endpoint is
\[
        u_1=N+\delta-v_0=u_0+2\delta.
\]
Thus after every pair of successive edges the moving endpoint increases by $2\delta$.  By induction the even endpoints along such a walk are $u_0,u_0+2\delta,u_0+4\delta,\ldots$.  Since $\delta>0$, this sequence cannot return to $u_0$ in a finite non-loop cycle.

We exclude sign-alternating paths of length at least two.  Reversing the order of $Z$ and interchanging the signs if necessary, the first two edges may be written
\[
        z_i+z_j=N-\delta,
        \qquad
        z_{i+1}+z_j=N+\delta .
\]
Then $z_{i+1}-z_i=2\delta$.  If a point of $Z$ lay strictly between $z_i$ and $z_{i+1}$, it would be closer than $\delta$ to $N-z_j$.  Hence $z_i,z_{i+1}$ are consecutive and
\[
        \lambda_i=\delta.
\]

Suppose first that the path has two edges.  By Lemma~\ref{lem:isolated-loop}, no vertex of this non-loop component carries a loop.  Since $z_i$ is a terminal vertex in the non-loop subgraph, its positive neighbor is absent from $Z$ in the full graph.  That positive neighbor would be
\[
        N+\delta-z_i=z_j+2\delta.
\]
Thus there is no point $z_j+2\delta$ of $Z$.  Applying the minimality of $\delta$ to the pair $(z_i,z_{j+1})$, if $j<s$, gives $\lambda_j\ge\delta$; if $j=s$, the right endpoint cap satisfies $c_R=N-z_s\ge\delta$ because $z_i=c_R-\delta$.  Equality in the internal case would give the missing positive neighbor, so the block or cap immediately to the right of $z_j$ has multiplicity at least $\delta$.

Similarly, since $z_{i+1}$ is terminal in the non-loop subgraph and carries no loop, its negative neighbor is absent from $Z$ in the full graph.  That neighbor would be $z_j-2\delta$.  If $j>0$, minimality applied to $(z_{i+1},z_{j-1})$ gives $\lambda_{j-1}\ge\delta$; if $j=0$, the left endpoint cap is at least $\delta$.  Thus the block or cap immediately to the left of $z_j$ also has multiplicity at least $\delta$.  The two adjacent blocks at $z_j$ therefore both have multiplicity at least $2$, contradicting Corollary~\ref{cor:no-adjacent-large}.

Now suppose the path has length at least three.  The third edge must have the opposite sign from the second one.  Thus, after possibly exchanging the two endpoints of the second edge, it has the form
\[
        z_{i+1}+z_{j-1}=N-\delta .
\]
Then $z_j-z_{j-1}=2\delta$, so the block immediately to the left of $z_j$ has multiplicity $\delta$.  The terminal condition at $z_i$, together with Lemma~\ref{lem:isolated-loop}, again excludes the positive neighbor $z_j+2\delta$ in the full graph; hence the block or endpoint cap immediately to the right of $z_j$ has multiplicity at least $\delta$.  These adjacent multiplicities are again both at least $2$, contradicting Corollary~\ref{cor:no-adjacent-large}.
\end{proof}

\begin{proposition}[Reduction to fully blocked extremal components]\label{prop:sealed-reduction}
If $\delta=\dist(Z,N-Z)\ge2$, then $\Gamma_\delta$ contains a fully blocked edge or an isolated fully blocked loop.  At the endpoints of such a component, any unblocked internal equality $\lambda=\delta$ would create an additional incident edge; hence the remaining internal blocking gaps at those endpoints have multiplicity strictly larger than $\delta$.
\end{proposition}

\begin{proof}
Choose $u,v\in Z$ with $|u+v-N|=\delta$; then $u$ and $v$ span an edge of $\Gamma_\delta$, with a loop allowed when $u=v$.  Hence $\Gamma_\delta$ is nonempty.  By Lemma~\ref{lem:degree-path}, every component of the non-loop-edge subgraph that contains an edge is a single edge.  By Lemma~\ref{lem:isolated-loop}, every loop component is isolated.  Lemma~\ref{lem:propagation} says that any unblocked internal equality $\lambda=\delta$ would create another edge.  Since the component is already a single edge or isolated loop, the remaining internal blocking gaps must have multiplicity $>\delta$.
\end{proof}

The exclusions below use one mass comparison repeatedly.  Passing to the adjacent cut value across an internal blocking gap of multiplicity $a$ and then reflecting produces a point in a $\delta$-core.  The coordinate identity of this reflected move gives an upper bound $a-\delta$ for the mass between the core containing the reflected point and the remaining blocker.  Sector filling gives a lower bound for the same mass.  Except for one endpoint-blocked case, these two bounds are incompatible; that one-sided endpoint case is handled by a second reflection in Lemma~\ref{lem:endpoint-blocked-growth}.

\begin{lemma}[Mass bounds after one reflected move]\label{lem:reflected-defect}
Assume $\delta=\dist(Z,N-Z)\ge2$.  The following upper bounds hold for a negative fully blocked edge or loop; the positive versions follow by Lemma~\ref{lem:reflection}.
\begin{enumerate}[label=(\roman*),leftmargin=*]
\item Let $z_i+z_j=N-\delta$, $i<j$, and suppose the left internal blocking gap has multiplicity $A=\lambda_i>\delta$.  Put $T=N-z_{i+1}$.  If $T$ lies in the $\delta$-core of the internal block $G_t=[z_t,z_{t+1}]$ with $t<j$, then
\[
        M(G_t,G_j)=\sum_{r=t+1}^{j-1}\lambda_r\le A-\delta.
\]
If instead $T$ lies in the left endpoint core, then
\[
        M(\mathcal B_L,G_j)\le A-\delta.
\]
\item Let $z_i+z_j=N-\delta$, $i<j$, and suppose the right internal blocking gap has multiplicity $B=\lambda_j>\delta$.  Put $T=N-z_{j+1}$.  If $T$ lies in the $\delta$-core of the internal block $G_t=[z_t,z_{t+1}]$ with $t<i$, then
\[
        M(G_t,G_i)=\sum_{r=t+1}^{i-1}\lambda_r\le B-\delta.
\]
If instead $T$ lies in the left endpoint core, then
\[
        M(\mathcal B_L,G_i)\le B-\delta.
\]
\item Let $2z_i=N-\delta$ be a loop with right internal blocking gap $A=\lambda_i>\delta$.  Put $T=N-z_{i+1}$.  If $T$ lies in the $\delta$-core of the internal block $G_t=[z_t,z_{t+1}]$, then
\[
        M(G_t,G_i)=\sum_{r=t+1}^{i-1}\lambda_r\le A-\delta.
\]
If $T$ lies in the left endpoint core, then
\[
        M(\mathcal B_L,G_i)\le A-\delta.
\]
\item Let $z_i+z_s=N-\delta$ be a one-sided endpoint-blocked edge, with right endpoint $z_s$ and left internal blocking gap $\alpha=\lambda_i>\delta$.  Put $T=N-z_{i+1}$.  If $T$ lies in the $\delta$-core of the internal block $G_t=[z_t,z_{t+1}]$, then
\[
        M(G_t,\mathcal B_R)=\sum_{r=t+1}^{s-1}\lambda_r\le \alpha-\delta.
\]
\end{enumerate}
Consequently, whenever sector filling forces the same mass to be larger than the displayed upper bound, the corresponding configuration is impossible.
\end{lemma}

\begin{proof}
We prove the displayed identities; each inequality then follows from the definition of a $\delta$-core.  In (i),
\[
        T=N-z_{i+1}=z_j+\delta-2A.
\]
If $T=z_t+\delta+2\ell$ with $0\le\ell\le h(G_t)-\delta$, then $z_j-z_t=2A+2\ell$, and hence
\[
        M(G_t,G_j)
        =\frac{z_j-z_{t+1}}2
        =A+\ell-h(G_t)
        \le A-\delta.
\]
If $T$ lies in the left endpoint core, then $T\le c_L-\delta$, so $z_j+\delta-2A\le c_L-\delta$ and
\[
        M(\mathcal B_L,G_j)=\frac{z_j-c_L}{2}\le A-\delta.
\]
The proof of (ii) is the same calculation with the right endpoint of the edge crossed: $T=N-z_{j+1}=z_i+\delta-2B$.

For (iii), $T=N-z_{i+1}=z_i+\delta-2A$.  If $T=z_t+\delta+2\ell$, then $z_i-z_t=2A+2\ell$, so
\[
        M(G_t,G_i)=A+\ell-h(G_t)\le A-\delta.
\]
If the reflected point lies in the left endpoint core, the same endpoint calculation gives $M(\mathcal B_L,G_i)\le A-\delta$.

For (iv), the edge equation gives $z_i=c_R-\delta$, and
\[
        T=N-z_{i+1}=z_s+\delta-2\alpha.
\]
If $T=z_t+\delta+2\ell$ with $0\le\ell\le h(G_t)-\delta$, then $z_s-z_t=2\alpha+2\ell$, and therefore
\[
        M(G_t,\mathcal B_R)
        =\frac{z_s-z_{t+1}}2
        =\alpha+\ell-h(G_t)
        \le\alpha-\delta.
\]
The final sentence follows by comparing these upper bounds with the corresponding sector-filling lower bounds.
\end{proof}

\subsection{Internal fully blocked configurations}

We exclude fully blocked configurations whose reflected point lies in the $\delta$-core of an internal gap.  By Lemma~\ref{lem:reflection} it is enough to write the negative cases.

In the edge case below we pass to an adjacent cut value across a blocking gap of minimum multiplicity.  The opposite side then gives the sharpest possible upper bound for the intervening mass.  Since the sector estimate is monotone in the blocking multiplicities, this choice turns the filling lower bound into a contradiction.

\begin{proposition}[Exclusion of internally blocked edges]\label{prop:internal-edge}
There is no internally fully blocked negative edge
\[
        z_i+z_j=N-\delta,
        \qquad i<j,
\]
with
\[
        A=\lambda_i>\delta,
        \qquad
        B=\lambda_j>\delta,
\]
for which the reflected point obtained after replacing one endpoint by the adjacent $Z$-point across a blocking gap of minimum multiplicity lies in the $\delta$-core of an internal gap.  If the two multiplicities are equal, either endpoint may be chosen; in the negative normal form used below we choose the left endpoint.  By reflection, the same conclusion holds for internally fully blocked positive edges.
\end{proposition}

\begin{proof}
First assume $A\le B$.  Thus the left gap has minimum multiplicity; in the tie case this is our chosen endpoint.  Replace $z_i$ by $z_{i+1}$ and put
\[
        T=N-z_{i+1}=z_j+\delta-2A.
\]
The middle mass between the two blocking gaps is
\[
        M_{\rm mid}=\sum_{r=i+1}^{j-1}\lambda_r=\frac{z_j-z_{i+1}}2.
\]
By Lemmas~\ref{lem:sector-filling} and~\ref{lem:theta-linear}, $M_{\rm mid}\ge\Theta(A,B)\ge A-1$.  Since $\delta\ge2$, this implies
\[
        z_j-z_{i+1}\ge 2A-\delta,
\]
and hence $T\in[z_{i+1},z_j)$.  By Lemma~\ref{lem:target-placement}, $T$ lies in an internal $\delta$-core or an endpoint $\delta$-core.  In the internal case considered in this proposition, $T$ lies in the $\delta$-core of some internal gap $[z_t,z_{t+1}]$, where $i+1\le t<j$.  Let $\gamma=\lambda_t$.

If $\gamma=\delta$, then the core is the single point $z_t+\delta$, so the newly created negative edge $z_{i+1}z_t$ satisfies $z_{i+1}+z_t=N-\delta$.  Propagation through the gap $[z_t,z_{t+1}]$ creates a positive edge incident to it.  If this negative edge is not a loop, we obtain a nontrivial path component of $\Gamma_\delta$, contradicting Lemma~\ref{lem:degree-path}; in the loop case, the loop has an incident non-loop edge, contradicting Lemma~\ref{lem:isolated-loop}.  Thus $\gamma>\delta$.

Write
\[
        T=z_t+\delta+2\ell,
        \qquad
        0\le \ell\le \gamma-\delta.
\]
Since also $T=z_j+\delta-2A$, we have
\[
        z_j-z_t=2A+2\ell.
\]
The mass between this reflected gap and the right blocking gap is
\[
        M_{\rm right}=\sum_{r=t+1}^{j-1}\lambda_r
        =\frac{z_j-z_{t+1}}2
        =A+\ell-\gamma
        \le A-\delta.
\]
On the other hand, sector filling between the gap of multiplicity $\gamma$ and the blocking gap of multiplicity $B$ gives
\[
        M_{\rm right}\ge\Theta(\gamma,B).
\]
Since $\gamma>\delta$ and $B\ge A>\delta$, Lemma~\ref{lem:absorption} gives
\[
        \Theta(\gamma,B)>A-\delta,
\]
a contradiction.

Now assume $A>B$.  Replace $z_j$ by $z_{j+1}$ and put
\[
        T'=N-z_{j+1}=z_i+\delta-2B.
\]
If $T'$ lies in the $\delta$-core of an internal gap $[z_t,z_{t+1}]$ with $t<i$, let $\gamma=\lambda_t$.  If $\gamma=\delta$, then $T'=z_t+\delta$ and the newly created edge $z_{j+1}z_t$ satisfies $z_{j+1}+z_t=N-\delta$; propagation through $[z_t,z_{t+1}]$ creates either a nontrivial alternating component or a loop with an incident non-loop edge.  Thus $\gamma>\delta$.  Writing
\[
        T'=z_t+\delta+2\ell,
        \qquad 0\le\ell\le\gamma-\delta,
\]
one obtains
\[
        \sum_{r=t+1}^{i-1}\lambda_r
        =B+\ell-\gamma
        \le B-\delta.
\]
Sector filling between $\gamma$ and $A$ gives a lower bound
\[
        \Theta(\gamma,A)>B-\delta,
\]
again a contradiction.
\end{proof}

\begin{proposition}[Exclusion of internally blocked loops]\label{prop:internal-loop}
There is no internally fully blocked loop whose target lies in the $\delta$-core of an internal gap.
\end{proposition}

\begin{proof}
Consider a negative loop $2z_i=N-\delta$ with right blocking gap $A=\lambda_i>\delta$.  Cross this blocker and put
\[
        T=N-z_{i+1}=z_i+\delta-2A.
\]
If $T$ lies in the $\delta$-core of an internal gap $[z_t,z_{t+1}]$, let $\gamma=\lambda_t$.  The equality $\gamma=\delta$ would create, by propagation from the newly created reflected edge, either a nontrivial sign-alternating component, contradicting Lemma~\ref{lem:degree-path}, or a loop with an incident non-loop edge, contradicting Lemma~\ref{lem:isolated-loop}; hence $\gamma>\delta$.  Write
\[
        T=z_t+\delta+2\ell,
        \qquad 0\le \ell\le\gamma-\delta.
\]
Then
\[
        z_i-z_t=2A+2\ell,
\]
and the mass between $[z_t,z_{t+1}]$ and the blocking gap $A$ is
\[
        \sum_{r=t+1}^{i-1}\lambda_r
        =A+\ell-\gamma
        \le A-\delta.
\]
Sector filling gives at least $\Theta(\gamma,A)$, and Lemma~\ref{lem:absorption} gives $\Theta(\gamma,A)>A-\delta$, a contradiction.  Positive loops follow from Lemma~\ref{lem:reflection}.
\end{proof}

\subsection{Endpoint fully blocked configurations}

We now treat the cases in which the reflected point lies in an endpoint core.  The next lemma reduces the internally blocked endpoint-core configurations to two negative left-cap forms.

\begin{lemma}[Endpoint-core reductions]\label{lem:endpoint-core-reductions}
It suffices to exclude the following two negative left-cap configurations; the reflected configurations are obtained by Lemma~\ref{lem:reflection}.
\begin{enumerate}
\item A negative fully blocked edge
\[
        z_i+z_j=N-\delta,
        \qquad i<j,
\]
with internal blocking gaps
\[
        A=\lambda_i>B=\lambda_j>\delta,
\]
such that the reflected point obtained by passing across the blocking gap of minimum multiplicity,
\[
        T=N-z_{j+1}=z_i+\delta-2B,
\]
lies in the left endpoint \(\delta\)-core.
\item A negative fully blocked loop
\[
        2z_i=N-\delta,
        \qquad A=\lambda_i>\delta,
\]
such that
\[
        T=N-z_{i+1}=z_i+\delta-2A
\]
lies in the left endpoint \(\delta\)-core.
\end{enumerate}
\end{lemma}

\begin{proof}
Consider a negative fully blocked edge with two internal blocking gaps.  If \(A\le B\), passing across the left blocking gap gives \(N-z_{i+1}\in [z_{i+1},z_j)\) by the calculation in Proposition~\ref{prop:internal-edge}.  Since this interval lies strictly between two cut values of \(Z\), the reflected point cannot belong to either endpoint core; Lemma~\ref{lem:target-placement} therefore places it in an internal core.  Proposition~\ref{prop:internal-edge} then gives a contradiction.  Hence a remaining endpoint-core edge with two internal blocking gaps must have \(A>B\), and passing across the minimum blocking gap \(B\) gives the displayed reflected point.  The reflected statements give the right-cap and positive-edge versions.

For a negative loop with an internal blocker, the same calculation starts from \(2z_i=N-\delta\) and the target \(N-z_{i+1}=z_i+\delta-2\lambda_i\).  If this target is internal, Proposition~\ref{prop:internal-loop} applies; otherwise the left-cap representative is the displayed loop.  Positive loops follow by reflection.
\end{proof}

\begin{lemma}[Endpoint-blocked cases with endpoint-core reflection]\label{lem:endpoint-blocked-endpoint-core}
No endpoint-blocked fully blocked edge has its target in an endpoint \(\delta\)-core.  Likewise, a loop whose only possible propagation is endpoint-blocked is impossible.
\end{lemma}

\begin{proof}
By reflection it is enough to consider a negative edge
\[
        z_i+z_s=N-\delta,
        \qquad i<s,
\]
whose propagation at \(z_i\) is internally blocked by \(A=\lambda_i>\delta\).  Crossing this blocker gives
\[
        T=N-z_{i+1}=z_s+\delta-2A<z_s.
\]
If \(T\) lies in an endpoint core, it must lie in the left endpoint core.  Put \(C=c_L\).  Then
\[
        z_s+\delta-2A=T\le C-\delta,
\]
so
\[
        \frac{z_s-C}{2}\le A-\delta.
\]
The mass between the left endpoint block \(C\) and the internal blocker \(A\) is
\[
        \frac{z_i-C}{2}
        \le \frac{z_s-C}{2}
        \le A-\delta.
\]
If \(C>\delta\), sector filling and Lemma~\ref{lem:absorption} give
\[
        \frac{z_i-C}{2}\ge \Theta(C,A)>A-\delta,
\]
a contradiction.  Hence \(C=\delta\).  The left endpoint core is then the single point \(0\), so \(T=0\) and \(z_{i+1}=N\).  This forces \(z_s=N\), and the edge equation gives \(z_i=-\delta\), impossible.

We rule out loops whose only possible propagation is endpoint-blocked.  For instance, if \(2z_s=N-\delta\), then
\[
        c_R=N-z_s=\frac{N+\delta}{2}.
\]
By Lemma~\ref{lem:endpoint-caps}, \(c_R+1=\gcd(k,a+1)\) divides \(k=N+1\), but
\[
        c_R+1=\frac{k+\delta+1}{2}>\frac k2.
\]
Thus \(c_R+1=k\), which would imply \(a+1\equiv0\pmod{k}\), contrary to the hypotheses.  The reflected endpoint-loop cases are identical.
\end{proof}

\begin{lemma}[Endpoint cap exactness]\label{lem:cap-exact}
Suppose a negative fully blocked edge in Lemma~\ref{lem:endpoint-core-reductions} remains after the internal-core exclusions.  Then it has the form
\[
        A=\lambda_i>B=\lambda_j>\delta,
\]
and the target
\[
        T=N-z_{j+1}=z_i+\delta-2B
\]
lies in the left endpoint \(\delta\)-core.  In this case \(c_L=\delta\).  The same conclusion holds for a remaining negative loop.
\end{lemma}

\begin{proof}
Let \(C=c_L\).  Since \(T\) lies in the left endpoint core, \(0\le T\le C-\delta\), so \(C\ge\delta\).  If \(C>\delta\), then
\[
        z_i+\delta-2B\le C-\delta,
\]
and hence
\[
        \frac{z_i-C}{2}\le B-\delta.
\]
The left side is the mass between the endpoint block \(C\) and the internal blocker \(A\).  Sector filling gives it at least \(\Theta(C,A)\).  Since \(C>\delta\) and \(A>B>\delta\), Lemma~\ref{lem:absorption} gives \(\Theta(C,A)>B-\delta\), a contradiction.  Thus \(C=\delta\).

For a negative loop, the target is \(T=N-z_{i+1}=z_i+\delta-2A\).  If \(T\in[0,C-\delta]\) and \(C>\delta\), then
\[
        \frac{z_i-C}{2}\le A-\delta,
\]
while sector filling gives at least \(\Theta(C,A)>A-\delta\).  Hence \(C=\delta\).
\end{proof}

\begin{lemma}[Endpoint normal forms]\label{lem:endpoint-normal-forms}
After the internal-core exclusions, any remaining negative left-cap endpoint-core configuration has one of the following forms.  For a fully blocked edge,
\begin{equation}\label{eq:endpoint-edge-normal}
        c_L=\delta,
        \qquad
        c_R=0,
        \qquad
        z_i=2B-\delta,
        \qquad
        z_j=N-2B,
        \qquad
        z_{j+1}=N,
\end{equation}
with \(A=\lambda_i>B=\lambda_j>\delta\).  For a fully blocked loop,
\begin{equation}\label{eq:endpoint-loop-normal}
        c_L=\delta,
        \qquad
        c_R=0,
        \qquad
        z_i=2A-\delta,
        \qquad
        z_{i+1}=N,
        \qquad
        N=4A-\delta,
\end{equation}
with \(A=\lambda_i>\delta\).
\end{lemma}

\begin{proof}
By Lemma~\ref{lem:cap-exact}, \(c_L=\delta\), so the left endpoint \(\delta\)-core is the singleton \(\{0\}\).  Thus the reflected point in Lemma~\ref{lem:cap-exact} is \(0\).

For an edge, \(N-z_{j+1}=0\), hence \(z_{j+1}=N\) and \(c_R=0\).  The equality \(z_i+\delta-2B=0\) gives \(z_i=2B-\delta\), and the edge equation gives \(z_j=N-\delta-z_i=N-2B\).  This is \eqref{eq:endpoint-edge-normal}.

For a loop, \(N-z_{i+1}=0\), so \(z_{i+1}=N\) and \(c_R=0\).  The equality \(z_i+\delta-2A=0\) gives \(z_i=2A-\delta\).  Since \(z_{i+1}=z_i+2A\), we obtain \(N=4A-\delta\), as in \eqref{eq:endpoint-loop-normal}.
\end{proof}

\begin{lemma}\label{lem:endpoint-ge3}
Endpoint fully blocked edges and endpoint fully blocked loops are impossible when \(\delta\ge3\).
\end{lemma}

\begin{proof}
For an edge in the normal form \eqref{eq:endpoint-edge-normal}, the mass between the left cap \(C=\delta\) and the blocking gap \(A\) is
\[
        \frac{z_i-C}{2}=B-\delta.
\]
Sector filling gives this mass at least \(\Theta(\delta,A)\).  By Lemma~\ref{lem:theta-bound},
\[
        \Theta(\delta,A)\ge\frac{(\delta-1)(A-1)}2.
\]
Since \(A>B>\delta\), we have \(B-\delta\le A-1-\delta\).  For \(\delta\ge3\),
\[
        \frac{(\delta-1)(A-1)}2>A-1-\delta,
\]
so the lower bound exceeds \(B-\delta\), a contradiction.

For a loop in the normal form \eqref{eq:endpoint-loop-normal}, the same calculation gives mass \(A-\delta\), while the same sector bound exceeds \(A-\delta\) for \(\delta\ge3\).
\end{proof}

We handle \(\delta=2\).  By Lemma~\ref{lem:endpoint-caps}, the normal form \(c_L=2\) and \(c_R=0\) is equivalent to
\[
        \gcd(k,a)=3,
        \qquad
        \gcd(k,a+1)=1.
\]
In the lattice notation, this is the case \(n=3\).  Thus \(k=3m\) and \(N=3m-1\).  The left endpoint ray is \(C=(1,0)\), with \(L(C)=m\) and \(H(C)=2\).  The example \(k=15,a=3\), worked out in Example~\ref{ex:width-three}, is a representative instance of this boundary configuration.  Figure~\ref{fig:width-three} illustrates the two families of sector points used in the next lemma.

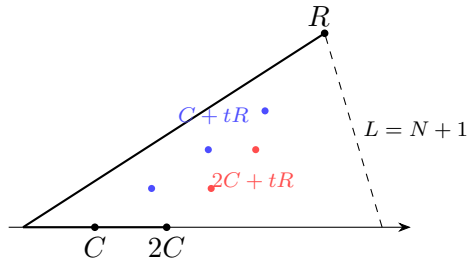
\begin{figure}[htbp]
\centering
\begin{tikzpicture}[scale=0.95,>=Stealth]
\coordinate (O) at (0,0);
\coordinate (C1) at (1.0,0);
\coordinate (C2) at (2.0,0);
\coordinate (R) at (4.2,2.7);
\draw[->] (-0.2,0)--(5.4,0);
\draw[thick] (O)--(R) node[pos=.98,above] {$R$};
\draw[thick] (O)--(C2);
\fill (C1) circle (1.5pt) node[below] {$C$};
\fill (C2) circle (1.5pt) node[below] {$2C$};
\fill (R) circle (1.5pt);
\foreach \u in {1,2,3}{\fill[blue!70] ($(C1)!\u/5!(R)+(0.15*\u,0)$) circle (1.4pt);}
\foreach \u in {1,2}{\fill[red!70] ($(C2)!\u/5!(R)+(0.18*\u,0)$) circle (1.4pt);}
\node[blue!70] at (2.65,1.55){\scriptsize $C+tR$};
\node[red!70] at (3.2,0.65){\scriptsize $2C+tR$};
\draw[dashed] (5.0,0)--(R) node[midway,right]{\scriptsize $L=N+1$};
\end{tikzpicture}
\caption{Schematic endpoint sector for \(c_L=2\).  When an endpoint ray has multiplicity two, the two families $C+tR$ and $2C+tR$ give the sharpened lower bound used in Lemma~\ref{lem:width-three}.}
\label{fig:width-three}
\end{figure}

The endpoint cap of multiplicity two requires a sharper boundary estimate.  Counting both columns $C+tR$ and $2C+tR$ gives the needed strengthening.  In the exceptional alignment $N=4\alpha-2$, the second endpoint column meets the boundary in exactly the range that accounts for the additional lower-bound term used below.

\begin{lemma}[Endpoint cap-sector filling for \(c_L=2\)]\label{lem:width-three}
Assume \(c_L=2\).  Let \(R\) be an internal primitive ray generator to the right of the left endpoint ray \(C\), and put \(\alpha=H(R)\).  In this normalization the endpoint cap consists of the two multiples \(C\) and \(2C\); the sector families \(C+tR\) and \(2C+tR\) give the terms used in the strengthened boundary count.  Then the mass between \(C\) and \(R\) satisfies
\[
        M(C,R)\ge \alpha-2.
\]
If in addition \(N=4\alpha-2\), then
\[
        M(C,R)\ge \alpha-1.
\]
The reflected statement holds when \(c_R=2\).
\end{lemma}

\begin{proof}
The left endpoint ray is \(C=(1,0)\), with \(L(C)=m\).  The hypothesis \(c_L=H(C)=2\) gives \(2m\le N<3m\).  Since \(N=mn-1\), this forces \(n=3\), and hence \(N=3m-1\).  Since \(H(R)=\alpha\), we have \(L(R)\le N/\alpha\).  Count the lattice points
\[
        C+tR,
        \qquad
        2C+tR
        \qquad(t\ge1)
\]
in the sector between \(C\) and \(R\).  The first type is certainly allowed when
\[
        t\le\frac{\alpha(N-m)}{N},
\]
and the second type is certainly allowed when
\[
        t\le\frac{\alpha(N-2m)}{N}.
\]
The two families are disjoint, because $C+tR=2C+t'R$ would imply $C=(t-t')R$, contradicting the linear independence of the endpoint ray and the internal ray.  As in Lemma~\ref{lem:sector-filling}, different lattice points on the same primitive ray are counted by the multiplicity of that ray, so these points inject into the multiplicity mass between \(C\) and \(R\).  Hence
\[
M(C,R)
\ge
F:=
\left\lfloor\frac{\alpha(N-m)}{N}\right\rfloor
+
\left\lfloor\frac{\alpha(N-2m)}{N}\right\rfloor.
\]
Since \(N=3m-1\), the sum of the two unfloored fractions is
\[
        \alpha-\frac{\alpha}{3m-1}.
\]
Because \(\alpha\le N=3m-1\), this sum is at least \(\alpha-1\).  More explicitly, if the two unfloored summands are \(x\) and \(y\), then each floor is strictly greater than the corresponding summand minus one, so
\[
        \lfloor x\rfloor+\lfloor y\rfloor>x+y-2\ge \alpha-3.
\]
Since \(F\) is an integer, \(F\ge\alpha-2\).

If \(N=4\alpha-2\), then \(3m-1=4\alpha-2\), so \(3m=4\alpha-1\).  This is the boundary alignment in which the two endpoint columns have complementary integer ranges along the boundary line: together they reach the integer total \(\alpha-1\).  Thus \(m\equiv1\pmod4\); write \(m=4q+1\).  Then \(\alpha=3q+1\) and \(N=12q+2\).  A direct calculation gives
\[
\left\lfloor\frac{\alpha(2m-1)}{N}\right\rfloor=2q,
\qquad
\left\lfloor\frac{\alpha(m-1)}{N}\right\rfloor=q.
\]
Therefore \(F=3q=\alpha-1\).  The right endpoint statement follows by applying the same argument to the reflected cut set \(N-Z\).
\end{proof}

\begin{proposition}[Exclusion of endpoint-core fully blocked configurations]\label{prop:endpoint-exclusion}
There are no endpoint fully blocked edges or endpoint fully blocked loops.
\end{proposition}

\begin{proof}
Endpoint-blocked endpoint-core cases are excluded by Lemma~\ref{lem:endpoint-blocked-endpoint-core}.  Thus it remains only to consider the two internally blocked normal forms above.  By Lemma~\ref{lem:endpoint-ge3}, we may assume \(\delta=2\).

For an endpoint fully blocked edge, \eqref{eq:endpoint-edge-normal} gives
\[
        z_i=2B-2,
        \qquad
        A=\lambda_i>B=\lambda_j>2.
\]
The mass between the left cap \(C\) and the blocking gap \(A\) is
\[
        M(\mathcal B_L,G_i)=\frac{z_i-c_L}{2}=\frac{2B-2-2}{2}=B-2.
\]
By Lemma~\ref{lem:width-three}, \(M(\mathcal B_L,G_i)\ge A-2\).  Hence \(B\ge A\), contradicting \(A>B\).

For an endpoint fully blocked loop, \eqref{eq:endpoint-loop-normal} gives
\[
        z_i=2A-2,
        \qquad
        N=4A-2.
\]
Thus
\[
        M(\mathcal B_L,G_i)=\frac{z_i-c_L}{2}=A-2.
\]
The strengthened part of Lemma~\ref{lem:width-three} gives \(M(\mathcal B_L,G_i)\ge A-1\), a contradiction.

The positive endpoint cases follow by Lemma~\ref{lem:reflection}.
\end{proof}

\begin{lemma}[Boundary equalities with internal sign gaps]\label{lem:boundary-created-internal-edge}
Assume \(\delta=\dist(Z,N-Z)\ge2\).  Let a point lying on the boundary of an internal \(\delta\)-core create an edge of \(\Gamma_\delta\) whose sign-propagation gaps are internal: either
\[
        z_p+z_t=N-\delta
\]
with right gaps at both endpoints internal, or
\[
        z_p+z_{t+1}=N+\delta
\]
with left gaps at both endpoints internal.  Then no such boundary equality can occur.
\end{lemma}

\begin{proof}
Consider the negative case; the positive case is the reflected argument.  If one of the internal gaps has multiplicity \(\delta\), Lemma~\ref{lem:propagation} creates an additional incident edge of the opposite sign.  This gives either a nontrivial sign-alternating component, contradicting Lemma~\ref{lem:degree-path}, or a loop with an incident non-loop edge, contradicting Lemma~\ref{lem:isolated-loop}.  If neither gap has multiplicity \(\delta\), then the edge or loop is internally blocked.  Lemma~\ref{lem:target-placement} places the reflected point obtained by passing across an internal blocking gap of minimum multiplicity in either an internal core or an endpoint core.  The internal alternatives are excluded by Propositions~\ref{prop:internal-edge} and~\ref{prop:internal-loop}, and the endpoint alternative is excluded by Proposition~\ref{prop:endpoint-exclusion}.
\end{proof}
\subsection{One-sided endpoint configurations}

It remains to handle the one-sided case: one endpoint of the extremal edge lies at an endpoint of $Z$, while the reflected point obtained from the internal adjacent gap lies in an internal $\delta$-core.  By reflection it suffices to consider the negative right-endpoint form below.  The reduction is
\[
\begin{gathered}
        z_i+z_s=N-\delta,
        \qquad
        \lambda_i=\alpha>\delta,\\
        C=c_R,
        \qquad
        \delta\le C<\alpha,
        \qquad
        c_L=0,
        \qquad
        \lambda_0=\cdots=\lambda_{i-1}=1,\\
        N-z_{i+1}\ \text{lies strictly inside a later gap of multiplicity }\gamma,\\
        \gamma\ge \delta+2+\frac{\delta(\alpha-1)}2,
        \qquad
        \alpha\ge \delta+1+\Theta(\gamma,C).
\end{gathered}
\]
The last two inequalities are incompatible.

\begin{lemma}[Prefix rigidity in a one-sided endpoint case]\label{lem:endpoint-blocked-initial}
Suppose
\[
        z_i+z_s=N-\delta,
        \qquad i<s,
\]
and suppose the propagation at $z_i$ is internally blocked by $\alpha=\lambda_i>\delta$.  Let $C=c_R=N-z_s$, so $z_i=C-\delta$, and assume that
\[
        T=N-z_{i+1}=z_s+\delta-2\alpha
\]
lies in the $\delta$-core of an internal gap $[z_t,z_{t+1}]$.  Then
\[
        \delta\le C<\alpha,
        \qquad
        c_L=0,
\]
and every gap before the $\alpha$-gap has multiplicity one:
\[
        \lambda_0=\lambda_1=\cdots=\lambda_{i-1}=1,
        \qquad
        z_r=2r\quad(0\le r\le i).
\]
\end{lemma}

\begin{proof}
Since $z_i=C-\delta$ is a cut value, $C\ge\delta$.  Write $\gamma=\lambda_t$ and
\[
        T=z_t+\delta+2\ell,
        \qquad
        0\le \ell\le \gamma-\delta.
\]
If $\gamma=\delta$, then $T=z_t+\delta$ and $z_{i+1}+z_t=N-\delta$.  Propagation through $[z_t,z_{t+1}]$ creates an edge of the opposite sign; in the non-loop case this gives a forbidden alternating component, and in the loop case it gives a loop with an incident non-loop edge.  Hence $\gamma>\delta$.

Lemma~\ref{lem:reflected-defect}(iv) gives the suffix upper bound
\[
        M(G_t,\mathcal B_R)=\sum_{r=t+1}^{s-1}\lambda_r
        =\alpha+\ell-\gamma
        \le \alpha-\delta.
\]
If $C\ge \alpha$, Lemma~\ref{lem:endpoint-sector} and Lemma~\ref{lem:absorption} give
\[
        M(G_t,\mathcal B_R)\ge \Theta(\gamma,C)>\alpha-\delta,
\]
a contradiction.  Thus $C<\alpha$.

Let $L_0=c_L=z_0$.  Since $z_i=C-\delta$, the total multiplicity between the left endpoint ray and the $\alpha$-gap is
\[
        \frac{C-\delta-L_0}{2}.
\]
If $L_0\ge2$, Lemma~\ref{lem:endpoint-sector} gives a lower bound at least
\[
        \Theta(L_0,\alpha)
        \ge \frac{(L_0-1)(\alpha-1)}2
        \ge \frac{C}{2},
\]
whereas the actual mass is at most $(C-\delta-L_0)/2\le C/2-2$.  This is impossible.  If $L_0=1$, then $c_L=1$ means $\gcd(k,a)=2$.  In the lattice notation this is $n=2$, so $N=2m-1$ and the left endpoint ray $E=(1,0)$ has linear form $L(E)=m=(N+1)/2$.  Since all cut values have parity $c_L$ and $z_i=C-\delta$, this case has $C\equiv\delta+1\pmod 2$.  Let $R$ be the primitive ray of multiplicity $\alpha$ corresponding to the $\alpha$-gap.  As in Lemma~\ref{lem:sector-filling}, the admissible points $E+rR$ inject into the multiplicity mass between $E$ and $R$.  Hence
\[
        M(E,R)
        \ge
        \left\lfloor \frac{\alpha(N-1)}{2N}\right\rfloor
        \ge
        \left\lfloor\frac{\alpha-1}{2}\right\rfloor
        \ge
        \left\lfloor\frac C2\right\rfloor.
\]
But the actual mass is $(C-\delta-1)/2\le(C-3)/2$, again impossible.  Therefore $c_L=0$.

We show that all gaps before the $\alpha$-gap have multiplicity one.  If one of these preceding gaps has multiplicity $\mu\ge2$, then the mass between that gap and the $\alpha$-gap is at most
\[
        \frac{C-\delta}{2}-\mu.
\]
Sector filling gives at least
\[
        \Theta(\mu,\alpha)
        \ge
        \frac{(\mu-1)(\alpha-1)}2
        \ge
        \frac C2,
\]
whereas $(C-\delta)/2-\mu\le C/2-3$.  Therefore all gaps before the $\alpha$-gap have multiplicity one.
\end{proof}

\begin{lemma}[Position of the first reflected point]\label{lem:endpoint-blocked-position}
Under the hypotheses of Lemma~\ref{lem:endpoint-blocked-initial}, the point $T=N-z_{i+1}$ does not lie in the $\delta$-core of the $\alpha$-gap itself.  It lies strictly inside the $\delta$-core of a later internal gap $[z_t,z_{t+1}]$, where $t>i+1$ and, with $p=i+1$,
\[
        T=z_t+\delta+2\ell,
        \qquad
        1\le \ell\le \gamma-\delta-1,
        \qquad
        \gamma=\lambda_t\ge\delta+2.
\]
\end{lemma}

\begin{proof}
By Lemma~\ref{lem:endpoint-blocked-initial}, $c_L=0$ and $C\ge\delta$; since $z_i=C-\delta$ is even, $C\equiv\delta\pmod2$.  The $\delta$-core of the $\alpha$-gap is
\[
        [C,\ C+2\alpha-2\delta].
\]
Let
\[
        M(G_i,\mathcal B_R)=\sum_{r=i+1}^{s-1}\lambda_r.
\]
Since $T=C+2M(G_i,\mathcal B_R)$, membership of $T$ in the $\alpha$-gap core would give $M(G_i,\mathcal B_R)\le \alpha-\delta$.  Sector filling gives $M(G_i,\mathcal B_R)\ge\Theta(\alpha,C)$.  If $(\delta,C)\ne(2,2)$, then either $\delta\ge3$ or $\delta=2$ and $C\ge4$, and Lemma~\ref{lem:theta-bound} gives $\Theta(\alpha,C)>\alpha-\delta$.  This is impossible.  If $\delta=C=2$, the reflected endpoint estimate of Lemma~\ref{lem:width-three} gives $M(G_i,\mathcal B_R)\ge\alpha-2$.  The upper bound from membership in the $\alpha$-gap core is $M(G_i,\mathcal B_R)\le\alpha-2$, so equality holds.  Hence
\[
        T=C+2M(G_i,\mathcal B_R)=2\alpha-2,
        \qquad
        z_{i+1}=C-\delta+2\alpha=2\alpha.
\]
Since $T=N-z_{i+1}$, this gives $N=4\alpha-2$.  The strengthened part of Lemma~\ref{lem:width-three} now gives $M(G_i,\mathcal B_R)\ge\alpha-1$, again impossible.  Thus $T$ lies in a later internal gap.

Write $T=z_t+\delta+2\ell$, $0\le\ell\le\gamma-\delta$.  If $t=i+1$, then the adjacent gaps $\alpha$ and $\gamma$ are both at least two, contradicting Corollary~\ref{cor:no-adjacent-large}; hence $t>i+1$.

If $\ell=0$, then $z_{i+1}+z_t=N-\delta$; if $\ell=\gamma-\delta$, then $z_{i+1}+z_{t+1}=N+\delta$.  In both cases the sign-propagation gaps are internal, namely the right gaps in the first case and the left gaps in the second.  Lemma~\ref{lem:boundary-created-internal-edge} excludes both boundary equalities.  Therefore $1\le\ell\le\gamma-\delta-1$, and in particular $\gamma\ge\delta+2$.
\end{proof}

\begin{lemma}[Second reflection inequalities]\label{lem:endpoint-blocked-growth}
Under the hypotheses and notation of Lemma~\ref{lem:endpoint-blocked-position}, the following two inequalities hold:
\begin{align}
        \gamma&\ge \delta+2+\frac{\delta(\alpha-1)}2,
        \label{eq:endpoint-growth-one}\\
        \alpha&\ge \delta+1+\Theta(\gamma,C).
        \label{eq:endpoint-growth-two}
\end{align}
They are incompatible.
\end{lemma}

\begin{proof}
Put $p=i+1$.  Here $p$ indexes the cut value $z_p=z_{i+1}$; the internal blocking gap remains $G_i=[z_i,z_{i+1}]$.  We now reflect the left endpoint $z_t$ of the $\gamma$-gap, rather than the original cut value $z_p$.  From $T=N-z_p=z_t+\delta+2\ell$ we obtain
\[
        N-z_t=z_p+\delta+2\ell.
\]
We first show that $z_p<N-z_t<z_t$.  The left inequality is clear.  For the right one, the mass between the $\alpha$-gap $G_i$ and the $\gamma$-gap $G_t$ is
\[
        M(G_i,G_t)=\sum_{r=i+1}^{t-1}\lambda_r=\sum_{r=p}^{t-1}\lambda_r=\frac{z_t-z_p}{2}.
\]
Sector filling gives
\[
        M(G_i,G_t)
        \ge \Theta(\alpha,\gamma)
        \ge \frac{(\alpha-1)(\gamma-1)}2
        \ge \frac{\delta(\gamma-1)}2.
\]
Since $\ell\le\gamma-\delta-1$, this is greater than $\ell+\delta/2$.  Hence $z_t-z_p>\delta+2\ell$, proving $N-z_t<z_t$.

By Lemma~\ref{lem:target-placement}, $N-z_t$ lies in the $\delta$-core of an internal gap $[z_r,z_{r+1}]$ with $p\le r<t$.  Write $\eta=\lambda_r$ and
\[
        N-z_t=z_r+\delta+2u,
        \qquad
        0\le u\le \eta-\delta.
\]
The boundary cases $u=0$ and $u=\eta-\delta$ create, respectively, an internal negative or positive edge whose sign-propagation gaps are internal.  Lemma~\ref{lem:boundary-created-internal-edge} excludes these boundary equalities.  Thus
\[
        1\le u\le\eta-\delta-1,
        \qquad
        \eta\ge\delta+2.
\]
Moreover,
\[
        z_r=z_p+2(\ell-u).
\]
If $r=p$, then the adjacent gaps $\alpha$ and $\eta$ are both large, impossible.  Hence $r>p$, so $1\le u<\ell$.  The mass between the $\alpha$-gap $G_i$ and the $\eta$-gap $G_r$ is
\[
        M(G_i,G_r)=\sum_{q=i+1}^{r-1}\lambda_q=\ell-u,
\]
and sector filling gives
\[
        \ell-u\ge\Theta(\alpha,\eta)
        \ge\frac{(\alpha-1)(\eta-1)}2
        \ge\frac{\delta(\alpha-1)}2.
\]
Using $u\ge1$ and $\ell\le\gamma-\delta-1$ gives \eqref{eq:endpoint-growth-one}.

The second inequality comes from the suffix after the $\gamma$-gap.  Since
\[
        z_t+\delta+2\ell=T=z_s+\delta-2\alpha,
\]
we have the suffix upper bound
\[
        M(G_t,\mathcal B_R)=\sum_{r=t+1}^{s-1}\lambda_r=\alpha+\ell-\gamma\le \alpha-\delta-1.
\]
Because $C\ge\delta>0$, endpoint sector filling applies to the $\gamma$-gap and the endpoint cap $C$, so
\[
        M(G_t,\mathcal B_R)\ge\Theta(\gamma,C).
\]
Combining the two bounds gives
\[
        \Theta(\gamma,C)\le \alpha-\delta-1,
\]
which is \eqref{eq:endpoint-growth-two}.

If $(\delta,C)\ne(2,2)$, then $C\equiv\delta\pmod2$ implies either $\delta\ge3$ or $\delta=2$ and $C\ge4$.  Combining \eqref{eq:endpoint-growth-one}, \eqref{eq:endpoint-growth-two}, and Lemma~\ref{lem:theta-bound} gives
\[
        \alpha\ge
        \delta+1+
        \frac{(\delta+1)(C-1)}2
        +R(\alpha-1),
        \qquad
        R=\frac{\delta(C-1)}4.
\]
Here $R>1$.  Rearranging yields
\[
        (R-1)\alpha
        \le
        R-
        \left(\delta+1+\frac{(\delta+1)(C-1)}2\right),
\]
This right-hand side is
\[
        -\left(\delta+1+\frac{(\delta+2)(C-1)}4\right),
\]
which is negative.  Since $R-1>0$ and $\alpha>0$, this is impossible.

It remains to treat the boundary alignment isolated in Lemma~\ref{lem:width-three}, namely $\delta=C=2$.  Then \eqref{eq:endpoint-growth-one} gives $\gamma\ge\alpha+3$.  The reflected endpoint estimate for the multiplicity-two cap gives
\[
        \sum_{r=t+1}^{s-1}\lambda_r\ge \gamma-2,
\]
while the suffix bound above gives the same sum at most $\alpha-3$.  Hence $\gamma\le\alpha-1$, a contradiction.
\end{proof}

\begin{proposition}[One-sided endpoint exclusion]\label{prop:endpoint-blocked-internal}
There is no negative edge
\[
        z_i+z_s=N-\delta,
        \qquad i<s,
\]
whose propagation at $z_i$ is internally blocked by $\lambda_i>\delta$ and for which $N-z_{i+1}$ lies in the $\delta$-core of an internal gap.  The reflected statement holds for positive endpoint-blocked edges at the left endpoint.
\end{proposition}

\begin{proof}
Lemmas~\ref{lem:endpoint-blocked-initial}, \ref{lem:endpoint-blocked-position}, and~\ref{lem:endpoint-blocked-growth} exclude the negative right-endpoint case.  The positive left-endpoint case follows by Lemma~\ref{lem:reflection}.
\end{proof}

\subsection{The reflection estimate and completion of the first family}

\begin{lemma}[Exhaustion of fully blocked components]\label{lem:sealed-exhaustion}
Assume \(\delta=\dist(Z,N-Z)\ge2\), and let a fully blocked edge or isolated fully blocked loop be produced by Proposition~\ref{prop:sealed-reduction}.  After applying the reflection symmetry of Lemma~\ref{lem:reflection}, choose an adjacent cut value as follows: across a blocking gap of minimum multiplicity when two internal blockers are present, across the unique internal blocker when exactly one internal blocker is present, and across the internal blocker of a loop when one exists.  If no internal adjacent step exists, the component is an endpoint-loop placement.

With this convention the fully blocked component falls into one of the alternatives in the following matrix.
\begin{center}
\small
\begin{tabular}{@{}>{\raggedright\arraybackslash}p{0.22\linewidth}>{\raggedright\arraybackslash}p{0.20\linewidth}>{\raggedright\arraybackslash}p{0.25\linewidth}>{\raggedright\arraybackslash}p{0.23\linewidth}@{}}
\textbf{component} & \textbf{blocked sides} & \textbf{reflected location after the chosen step} & \textbf{excluded by}\\ \hline
non-loop edge & two internal & internal core & Proposition~\ref{prop:internal-edge}\\
non-loop edge & two internal & endpoint core & Proposition~\ref{prop:endpoint-exclusion}\\
non-loop edge & one internal, one endpoint & internal core & Proposition~\ref{prop:endpoint-blocked-internal}\\
non-loop edge & one internal, one endpoint & endpoint core & Lemma~\ref{lem:endpoint-blocked-endpoint-core}\\
loop & internal blocker & internal core & Proposition~\ref{prop:internal-loop}\\
loop & internal blocker & endpoint core & Proposition~\ref{prop:endpoint-exclusion}\\
endpoint loop & no internal step & endpoint placement & Proposition~\ref{prop:endpoint-exclusion}
\end{tabular}
\end{center}
Positive signs and the opposite endpoint are obtained by applying Lemma~\ref{lem:reflection}; the reflection reverses the block order, interchanges \(\mathcal B_L\) and \(\mathcal B_R\), sends \(G_i\) to \(G_{s-1-i}\), and preserves the sector-filling inequalities after reversing the slope order.
\end{lemma}

\begin{proof}
For a negative edge \(z_i+z_j=N-\delta\), the possible propagation at either endpoint is to the right; for a positive edge it is to the left.  This is Definition~\ref{def:blocked} and Lemma~\ref{lem:propagation}.  If an endpoint has an internal blocking gap, passing to the adjacent cut value across that gap gives another point of \(Z\), and Lemma~\ref{lem:target-placement} places its reflection either in an internal core or in an endpoint core.  The number of internal blocked sides is therefore two, one, or zero; the two- and one-sided cases give the non-loop edge rows of the table, while the zero-sided case is excluded as a non-loop edge in the next paragraph.

When two internal blocking gaps are present, choosing one of minimum multiplicity gives precisely the normal form used in Proposition~\ref{prop:internal-edge}; if the chosen reflected point lands in an endpoint core, the endpoint reductions and Proposition~\ref{prop:endpoint-exclusion} apply.  When exactly one side is endpoint-blocked and the other is internal, the internal placement is Proposition~\ref{prop:endpoint-blocked-internal}, while the endpoint placement is Lemma~\ref{lem:endpoint-blocked-endpoint-core}.  If there is no internal adjacent step, then a non-loop edge cannot occur: in the negative case both possible right moves would have to be at $z_s$, and in the positive case both possible left moves would have to be at $z_0$.  Thus this remaining row is an endpoint loop, hence an endpoint placement.

A loop has the same alternatives with the two endpoints identified: either there is an internal blocker, giving an internal-core or endpoint-core placement, or there is no internal step and the loop is endpoint-blocked.  Lemma~\ref{lem:reflection} reduces the positive and opposite-end cases to the listed negative representatives and preserves the displayed exclusions.
\end{proof}

\begin{theorem}[Cut-reflection theorem]\label{thm:near-antipodal}
For $D(k,a)=\Cay(\Z_k;a,a+1)$ with $a\not\equiv0,-1\pmod{k}$, the Hamiltonian cut set $Z$ satisfies
\[
        \dist(Z,N-Z)\le1,
        \qquad N=k-1.
\]
\end{theorem}

\begin{proof}
Assume $\delta=\dist(Z,N-Z)\ge2$.  Choose $u,v\in Z$ with $|u+v-N|=\delta$; then $\Gamma_\delta$ is nonempty.  Proposition~\ref{prop:sealed-reduction} gives a fully blocked edge or an isolated fully blocked loop.  Apply Lemma~\ref{lem:sealed-exhaustion}.  Its case matrix leaves only the following exclusion mechanisms.  The displayed mass comparisons are representative; the other orientations are obtained by interchanging the endpoints or applying reflection:
\begin{center}
\small
\begin{tabular}{@{}>{\raggedright\arraybackslash}p{0.30\linewidth}>{\raggedright\arraybackslash}p{0.36\linewidth}>{\raggedright\arraybackslash}p{0.25\linewidth}@{}}
\textbf{canonical alternative} & \textbf{representative mass comparison} & \textbf{exclusion}\\ \hline
Two internal blockers, edge & $M(G_t,G_j)\le A-\delta<\Theta(h(G_t),B)$ & Proposition~\ref{prop:internal-edge}\\
Internal blocker, loop & $M(G_t,G_i)\le A-\delta<\Theta(h(G_t),A)$ & Proposition~\ref{prop:internal-loop}\\
One endpoint cap and one internal blocker & prefix rigidity, then second reflection & Proposition~\ref{prop:endpoint-blocked-internal}\\
Endpoint-core placement, or no internal move & $C=\delta$, then cap-sector contradiction & Proposition~\ref{prop:endpoint-exclusion}
\end{tabular}
\end{center}
Lemma~\ref{lem:sealed-exhaustion} shows that these mechanisms cover every fully blocked component.  Hence no component of $\Gamma_\delta$ can exist, contradicting the choice of $u,v$.  Therefore $\dist(Z,N-Z)\le1$.
\end{proof}

\begin{corollary}[Parity-sharp form]\label{cor:parity-sharp}
For $D(k,a)=\Cay(\Z_k;a,a+1)$ with $a\not\equiv0,-1\pmod{k}$, let $Z$ be the Hamiltonian cut set and $N=k-1$.  Then
\[
        \dist(Z,N-Z)=
        \begin{cases}
        0,& k\text{ odd},\\
        1,& k\text{ even}.
        \end{cases}
\]
Equivalently, if $k$ is odd then there are $d,e\in Z$ with $d+e=k-1$, while if $k$ is even exact reflection is forbidden by parity and the cut-reflection theorem gives a pair with $d+e\in\{k-2,k\}$.
\end{corollary}

\begin{proof}
All elements of $Z$ have the same parity, since consecutive cut values differ by $2\lambda_r$.  If $k$ is odd, then $N=k-1$ is even, so $Z$ and $N-Z$ have the same parity.  Theorem~\ref{thm:near-antipodal} gives distance at most one, hence distance zero.  If $k$ is even, then $N$ is odd, so $Z$ and $N-Z$ have opposite parity and cannot intersect.  The distance is therefore exactly one.
\end{proof}

Theorem~\ref{thm:near-antipodal} is a near-antipodal pair statement.  It gives a pair of cut values whose sum is within one of $N$; it does not assert that the whole set $Z$ is close to $N-Z$ in Hausdorff distance.

\begin{example}[Sharpness of the reflection bound]\label{ex:first-family}
Continuing Example~\ref{ex:lattice-10-4}, we have $N=9$, and the cut permutations give
\[
        Z=\{1,3,5\},
        \qquad
        N-Z=\{8,6,4\}.
\]
Here $k$ is even, so Corollary~\ref{cor:parity-sharp} predicts that exact reflection is impossible and distance one is the best possible.  The count identities
\[
        3+5=8=k-2,
        \qquad
        5+5=10=k
\]
then give the input required by the count criterion.  Figure~\ref{fig:near-antipodal-example} shows the reflected sets on the number line.
\end{example}

\begin{figure}[H]
\centering
\begin{tikzpicture}[x=0.63cm,y=0.6cm]
\draw[->] (-0.2,0)--(9.7,0);
\foreach \x in {0,1,...,9}{\draw (\x,0.09)--(\x,-0.09) node[below]{\scriptsize \x};}
\foreach \x in {1,3,5}{\fill[blue] (\x,0) circle (2.2pt);}
\foreach \x in {4,6,8}{\draw[red,thick] (\x,0) circle (2.5pt);}
\draw[<->,thick] (4,0.55)--(5,0.55) node[midway,above]{\scriptsize distance $1$};
\node[blue] at (2.0,0.85){\scriptsize $Z$};
\node[red] at (7.2,0.85){\scriptsize $N-Z$};
\end{tikzpicture}
\caption{The Hamiltonian cut values $Z=\{1,3,5\}$ for $\Cay(\Z_{10};4,5)$ and their reflection $N-Z=\{8,6,4\}$ with $N=9$.}
\label{fig:near-antipodal-example}
\end{figure}
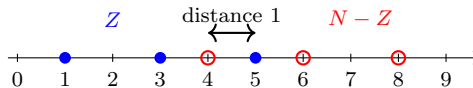

\begin{example}[Endpoint cap of multiplicity two]\label{ex:width-three}
For $k=15$ and $a=3$, one has $\gcd(k,a)=3$ and $\gcd(k,a+1)=1$, hence $c_L=2$ and $c_R=0$.  Direct computation of the cut permutations gives
\[
        Z=\{2,4,6,8,14\},
        \qquad N=14.
\]
The exact reflected pair $6+8=14$ shows how the endpoint cap of multiplicity two is forced to produce a count pair.  Lemma~\ref{lem:width-three} is the general estimate behind this small example.
\end{example}

\begin{proof}[Proof of Theorem~\ref{thm:first-family}]
By Theorem~\ref{thm:near-antipodal}, there exist $d,e\in Z$ (possibly equal) such that
\[
        |d-(N-e)|\le1.
\]
Equivalently,
\[
        d+e\in\{N-1,N,N+1\}=\{k-2,k-1,k\}.
\]
Since every element of $Z$ is realized by a Hamiltonian cut path, the chosen $d,e\in Z$ give Hamiltonian paths $P_d,P_e$ with
\[
        \delta_b(P_d)=d,
        \qquad
        \delta_b(P_e)=e.
\]
Theorem~\ref{input:count} now gives two arc-disjoint Hamiltonian paths.
\end{proof}

\begin{corollary}[Finite enumeration of cut values]\label{cor:cut-search}
For the family $\Cay(\Z_k;a,a+1)$, a pair of cut values $d,e\in Z$ with $d+e\in\{k-2,k-1,k\}$ can be found by testing the $k$ permutations $\Phi_d$ and forming the set $Z$.  A direct cycle test for each permutation gives an $O(k^2)$ enumeration of cut values.  Given such a pair, the construction in the proof of \cite[Proposition~4.3]{DMM} produces two arc-disjoint Hamiltonian paths from the corresponding Hamiltonian cut paths.
\end{corollary}

\section{The second cyclic family}\label{sec:second-family}

\begin{theorem}\label{thm:second-family}
Let $a\ge1$, put $M=2a+1$, and let $k=ML$ with $L\ge2$.  Then
\[
        \Cay(\Z_k;-a,a+1)
\]
has two arc-disjoint Hamiltonian paths.  Moreover, one path is obtained by deleting a single arc from a Hamiltonian cycle arising from a quotient-position rule with $|S|=a+2$; the other is obtained from the complementary cover, with one splice required exactly when $L$ is even.
\end{theorem}

\begin{proof}
Put
\[
        A=-a,
        \qquad
        B=a+1.
\]
Then
\[
        B-A=M.
\]
Since $L\ge2$, the two generators are distinct modulo $k$.  Moreover,
\[
        \langle A,B\rangle=\langle B,M\rangle=\Z_k,
\]
because $\gcd(B,M)=\gcd(a+1,2a+1)=1$.

Modulo $M$, the generators $A$ and $B$ are equal, and this common residue is a unit.  Figure~\ref{fig:quotient-fiber} gives the quotient-fiber picture.  Let $t(x)\in\Z_M$ be the quotient coordinate defined by
\[
        x\equiv t(x)(a+1)\pmod M.
\]
Then both $A$-arcs and $B$-arcs increase $t$ by one.

\begin{figure}[H]
\centering
\begin{tikzpicture}[x=1.0cm,y=0.62cm,>=Stealth]
\foreach \x in {0,1,2,3,4}{
  \foreach \y in {0,1,2}{\draw[gray!45] (\x,\y) circle (1.1pt);}
}
\foreach \x in {0,1,2,3}{
  \draw[->,blue,thick] (\x,0.18)--(\x+1,0.18);
  \draw[->,red,dashed,thick] (\x,1.05)--(\x+1,1.05);
}
\draw[->,blue,thick] (4,0.18) .. controls (4.55,0.42) and (-0.55,0.42) .. (0,0.18);
\draw[->,red,dashed,thick] (4,1.05) .. controls (4.6,1.35) and (-0.6,1.35) .. (0,1.05);
\draw[decorate,decoration={brace,amplitude=4pt}] (-0.28,0)--(-0.28,2) node[midway,left=5pt] {\scriptsize fiber};
\draw[decorate,decoration={brace,mirror,amplitude=4pt}] (0,-0.25)--(4,-0.25) node[midway,below=5pt] {\scriptsize quotient cycle};
\node[blue] at (2,-0.75) {\scriptsize $A$-arcs};
\node[red] at (2,1.85) {\scriptsize $B$-arcs};
\end{tikzpicture}
\caption{Quotient-fiber view of the second cyclic family.  Both generators advance one step in the quotient; one full quotient round gives fiber shift $-1$ for $P$ and $+2$ for the complementary cover $Q$.}
\label{fig:quotient-fiber}
\end{figure}
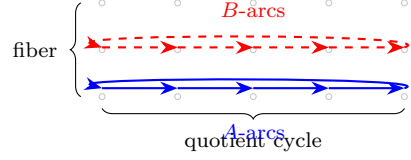

For a subset $S\subseteq\Z_M$, define a map $F_S$ of $\Z_k$ by
\[
F_S(x)=
\begin{cases}
        x+A, & t(x)\in S,\\
        x+B, & t(x)\notin S.
\end{cases}
\]
This is a skew product over the quotient cycle: in coordinates $(t,j)$ it has the form $(t,j)\mapsto(t+1,j+\epsilon(t))$, where the fiber increment depends on whether $t\in S$.  It is a permutation: if $y$ has quotient coordinate $t(y)$, then its predecessor must have quotient coordinate $t(y)-1$, and it is uniquely $y-A$ when $t(y)-1\in S$ and $y-B$ otherwise.  Since the quotient motion is one $M$-cycle, the cycle structure of $F_S$ is determined by the return shift on the fiber.

If $|S|=x$, then one quotient round uses $x$ arcs labelled $A$ and $M-x$ arcs labelled $B$, so the total displacement is
\[
        xA+(M-x)B=M(a+1-x).
\]
On the fiber $M\Z_k\cong\Z_L$, the one-round shift is therefore $a+1-x$.  In quotient-fiber coordinates the map has the form $(t,j)\mapsto(t+1,j+\epsilon(t))$, and one full quotient round sends $j$ to $j+a+1-x$.  Since the quotient motion is a single $M$-cycle, the cycles of $F_S$ are the orbits of this return shift on the fiber, and hence their number is $\gcd(L,a+1-x)$.

For a general choice $|S|=x$, the complementary cover has shift $x-a=1-(a+1-x)$, so the two return shifts always sum to $1$.  To obtain a Hamiltonian cycle for every $L$, choose the quotient positions so that the first return shift is a unit.  Thus choose any subset $S\subseteq\Z_M$ with $|S|=a+2$, and put $C_S=F_S$.  The cycle structure of $C_S$ depends on $S$ through the return shift $a+1-|S|$ on the fiber.  Then $C_S$ has fiber shift $-1$.  Since the quotient motion is a single $M$-cycle and the return shift $-1$ is a unit on the fiber $\Z_L$, the permutation $C_S$ is a Hamiltonian cycle; call this cycle $P$.  Let $Q=F_{\Z_M\setminus S}$ be the complementary cycle cover.  It is arc-disjoint from $P$, since at each vertex it uses the other generator.  The number of quotient positions where $Q$ uses $A$ is
\[
        M-(a+2)=a-1,
\]
so the fiber shift of $Q$ is
\[
        a+1-(a-1)=2.
\]
Thus $Q$ has $\gcd(L,2)$ cycles.

If $L$ is odd, $Q$ is a Hamiltonian cycle, and deleting one arc from each of $P$ and $Q$ gives two arc-disjoint Hamiltonian paths.

Assume now that $L$ is even.  Then $Q$ has two directed cycles, say $Q_0$ and $Q_1$.  Since $P$ is a Hamiltonian cycle, some arc $e=(u,v)$ of $P$ goes from one $Q$-cycle to the other.  Indeed, if no such arc existed, then both nonempty proper sets $V(Q_0)$ and $V(Q_1)$ would be invariant under the successor permutation of $P$, contradicting the fact that this permutation has a single orbit on all of $\Z_k$.

Assume $u\in V(Q_0)$ and $v\in V(Q_1)$.  Removing the outgoing $Q$-arc at $u$ turns the $Q_0$-cycle into a path ending at $u$; removing the incoming $Q$-arc at $v$ turns the $Q_1$-cycle into a path starting at $v$.  Inserting the arc $u\to v$ concatenates these two paths in the correct direction, producing one Hamiltonian path.  Remove the same arc $u\to v$ from the Hamiltonian cycle $P$.  The result is another Hamiltonian path.  The inserted arc is the arc removed from $P$, and before insertion it was not an arc of $Q$, because $P$ and $Q$ used complementary generators at every vertex.  Hence the two final paths are arc-disjoint.
\end{proof}

\begin{example}\label{ex:second-family}
Take $a=1$, $M=3$, $k=6$, and $L=2$.  Then $A=-1\equiv5$ and $B=2$ modulo $6$.  The quotient coordinate is determined by $x\equiv 2t(x)\pmod 3$, so the successive vertices $0,5,4,3,2,1$ have quotient coordinates $0,1,2,0,1,2$.  Choosing $S=\Z_3$ gives the Hamiltonian cycle
\[
        P:
        0\to5\to4\to3\to2\to1\to0.
\]
The complement uses only $B$-arcs and has two cycles
\[
        0\to2\to4\to0,
        \qquad
        1\to3\to5\to1.
\]
The arc $0\to5$ of $P$ crosses between these two cycles.  Removing $0\to2$ and $3\to5$ from the complement and inserting $0\to5$ gives
\[
        2\to4\to0\to5\to1\to3,
\]
a Hamiltonian path.  Removing $0\to5$ from $P$ gives the second Hamiltonian path.
\end{example}

\FloatBarrier
\section{The finite abelian two-generator theorem}

\begin{proof}[Proof of Theorem~\ref{thm:main}]
The reduction theorem of Darijani--Miraftab--Witte Morris leaves exactly the two cyclic assertions stated in Theorem~\ref{input:reduction}.  The first, the family $\Cay(\Z_k;a,a+1)$, is Theorem~\ref{thm:first-family}; its proof obtains a pair of Hamiltonian cut values whose sum lies in $\{k-2,k-1,k\}$ from the cut-reflection theorem and then applies Theorem~\ref{input:count}.  The second, the family $\Cay(\Z_k;-a,a+1)$, is Theorem~\ref{thm:second-family}, proved directly by the quotient--fiber cycle and the complementary cover.  These two results are precisely the cyclic assertions required by the reduction theorem.  Hence Theorem~\ref{thm:main} follows.
\end{proof}

\FloatBarrier
\section{The product of three directed cycles}\label{sec:threecycles}

Darijani--Miraftab--Witte Morris proved the two-factor case and the cases with at least four directed-cycle factors.  Theorem~\ref{thm:threecycles} proves the remaining three-factor case; together these results prove Corollary~\ref{cor:all-cycle-products}.

The lifting construction below alternates translated copies of two Hamiltonian paths in consecutive layers.  The condition $P\cap(Q+\gamma)=\varnothing$ gives the horizontal disjointness needed when the relative translation is $\gamma$, while the two terminal inequalities ensure that the vertical arcs leaving a layer have distinct tails.

\begin{definition}[Strongly switchable pair]\label{def:strong-switchable}
Let $D=\Cay(G;s_1,s_2)$ be a directed Cayley digraph, and let $P,Q$ be arc-disjoint Hamiltonian paths in $D$.  Write their initial and terminal vertices as
\[
        \iota_P,\tau_P,
        \qquad
        \iota_Q,\tau_Q.
\]
Set
\[
        \alpha=\tau_P-\iota_Q,
        \qquad
        \beta=\tau_Q-\iota_P,
        \qquad
        \gamma=\alpha-\beta.
\]
The ordered pair $(P,Q)$ is strongly switchable if
\[
        P\cap(Q+\gamma)=\varnothing
\]
as arc sets, and
\[
        \tau_P\ne\tau_Q,
        \qquad
        \tau_P\ne\tau_Q+
        \gamma.
\]
\end{definition}

\begin{lemma}[Strongly switchable lifting]\label{lem:strong-lifting}
If a finite directed Cayley digraph $D$ has a strongly switchable pair, then $D\Boxprod\vec C_\ell$ has two arc-disjoint Hamiltonian paths for every $\ell\ge2$.
\end{lemma}

Figure~\ref{fig:strong-lifting} shows the layer pattern used in the proof.

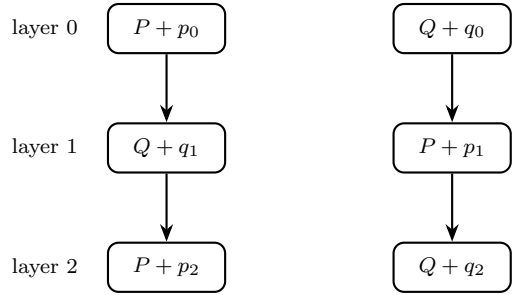
\begin{figure}[H]
\centering
\begin{tikzpicture}[node distance=0.9cm, >=Stealth]
\node[layerbox] (p0) {$P+p_0$};
\node[layerbox, right=2.2cm of p0] (q0) {$Q+q_0$};
\node[layerbox, below=of p0] (q1) {$Q+q_1$};
\node[layerbox, below=of q0] (p1) {$P+p_1$};
\node[layerbox, below=of q1] (p2) {$P+p_2$};
\node[layerbox, below=of p1] (q2) {$Q+q_2$};
\draw[->, thick] (p0)--(q1);
\draw[->, thick] (q0)--(p1);
\draw[->, thick] (q1)--(p2);
\draw[->, thick] (p1)--(q2);
\node[left=0.25cm of p0] {\scriptsize layer $0$};
\node[left=0.25cm of q1] {\scriptsize layer $1$};
\node[left=0.25cm of p2] {\scriptsize layer $2$};
\end{tikzpicture}
\caption{Layer pattern in the strong switchable lifting.  The two Hamiltonian paths alternate the translated copies of $P$ and $Q$; the recurrence for $p_i,q_i$ ensures that the displayed vertical arcs exist.}
\label{fig:strong-lifting}
\end{figure}

\begin{proof}
Let $P,Q,\alpha,\beta,\gamma$ be as in Definition~\ref{def:strong-switchable}.  Define translations $p_i,q_i\in G$ by
\[
        p_0=q_0=0,
        \qquad
        q_{i+1}=p_i+\alpha,
        \qquad
        p_{i+1}=q_i+\beta.
\]
Then $q_i-p_i$ is $0$ for even $i$ and $\gamma$ for odd $i$.

In layer $i$ of $D\Boxprod\vec C_\ell$, place the paths $P+p_i$ and $Q+q_i$.  The first Hamiltonian path alternates
\[
        P+p_0,
        Q+q_1,
        P+p_2,
        Q+q_3,
        \ldots,
\]
and the second alternates
\[
        Q+q_0,
        P+p_1,
        Q+q_2,
        P+p_3,
        \ldots.
\]
For $0\le i\le \ell-2$, the recurrence gives the required vertical arcs from layer $i$ to layer $i+1$:
\[
        \tau_P+p_i\to \iota_Q+q_{i+1},
        \qquad
        \tau_Q+q_i\to \iota_P+p_{i+1}.
\]
No vertical arc from layer $\ell-1$ to layer $0$ is used.  Each layer contributes a translate of one of $P$ and $Q$ to each constructed path, hence every vertex in that layer is visited once by each path; the vertical arcs join the terminal of the path in layer $i$ to the initial vertex of the path in layer $i+1$.  Therefore the two paths are Hamiltonian paths of $D\Boxprod\vec C_\ell$.

Horizontal arcs are disjoint in every layer because $q_i-p_i\in\{0,\gamma\}$ and $P\cap Q=P\cap(Q+\gamma)=\varnothing$.  Horizontal arcs use directions in $D$, whereas vertical arcs use the $\vec C_\ell$ direction, so no horizontal arc can coincide with a vertical arc.  The two vertical arcs leaving a fixed layer have distinct tails because
\[
        \tau_P+p_i\ne\tau_Q+q_i
\]
whenever $q_i-p_i\in\{0,\gamma\}$, by strong switchability.  Therefore the two constructed spanning paths are arc-disjoint.
\end{proof}

\begin{remark}
The lifting lemma is independent of the special structure of two-cycle products.  Any finite Cayley digraph that carries a strongly switchable pair can be lifted through a directed-cycle factor in the same way.  Thus strong switchability is a sufficient local condition for lifting two arc-disjoint Hamiltonian paths through a directed-cycle factor.
\end{remark}

For instance, when $\ell=3$, the two lifted paths have layer pattern
\[
        (P+p_0)\to(Q+q_1)\to(P+p_2),
        \qquad
        (Q+q_0)\to(P+p_1)\to(Q+q_2).
\]
The equations $q_{i+1}=p_i+\alpha$ and $p_{i+1}=q_i+\beta$ are the endpoint equations that supply the two vertical arcs between consecutive layers.

\begin{lemma}[Uniform-coset shift compatibility]\label{lem:uniform-coset-shift}
Let $D=\Cay(G;a,b)$, let $w=a-b$, and let $C$ be a coset of $\langle w\rangle$.  Suppose that $P$ uses one fixed label on every tail in $C$, while $Q$ uses the other fixed label on every tail in $C$.  If $\gamma\in\langle w\rangle$, then $P$ and $Q+\gamma$ have no common arc with tail in $C$.  Consequently, for the paths supplied by Lemma~\ref{input:twocycle-structure}, $P$ and $Q+\gamma$ have no common arc outside the terminal coset of $P$.
\end{lemma}

\begin{proof}
Translation by $\gamma$ preserves each coset of $\langle w\rangle$.  Thus the arcs of $Q+\gamma$ with tails in $C$ still have tails in $C$, and they have the same label as the corresponding arcs of $Q$ on $C$.  By hypothesis this label is opposite to the label used by $P$ on $C$.  Two Cayley arcs with a common tail are equal only when their labels are equal.  Hence no arc of $P$ on $C$ equals an arc of $Q+\gamma$ on $C$.

For the final assertion, apply the first part on every nonterminal coset, using the uniform nonterminal-coset statement in Lemma~\ref{input:twocycle-structure}.
\end{proof}

\begin{lemma}[Two-cycle products provide strongly switchable pairs]\label{lem:two-cycle-strong}
Let $D=\vec C_m\Boxprod\vec C_n$.  If $D$ has no Hamiltonian directed cycle, then $D$ has a strongly switchable pair of arc-disjoint Hamiltonian paths.
\end{lemma}

\begin{proof}
Write $D=\Cay(\Z_m\times\Z_n;a,b)$ and $w=a-b$.  By Lemma~\ref{input:twocycle-structure}, choose arc-disjoint Hamiltonian paths $P,Q$ with the endpoint rigidity and terminal-coset structure stated there.  Let $\iota_P,\tau_P$ and $\iota_Q,\tau_Q$ be their endpoints.  The endpoint rigidity gives
\[
        \iota_Q\in\{\tau_P+a,\tau_P+b\},
        \qquad
        \tau_Q\in\{\iota_P-a,\iota_P-b\}.
\]
Thus $\gamma\in\{0,\pm w\}$.

Let $L=\ord(w)$ and write the terminal coset as
\[
        x_r=\tau_P+rw,
        \qquad r\in\Z_L.
\]
The terminal-coset interval description says that for some $d$,
\[
        P_b=\{x_1,\ldots,x_d\},
        \qquad
        P_a=\{x_{d+1},\ldots,x_{L-1}\}.
\]
The endpoint formula from Lemma~\ref{input:twocycle-structure} gives
\[
        \iota_P=\tau_P+a+dw.
\]
If $d=0$, then $\iota_P=\tau_P+a$, so the terminal vertex of $P$ is joined to its initial vertex by an $a$-arc.  If $d=L-1$, then
\[
        \iota_P=\tau_P+a+(L-1)w=\tau_P+a-w=\tau_P+b,
\]
so the terminal vertex is joined to the initial vertex by a $b$-arc.  In either case $P$ closes to a Hamiltonian directed cycle, contrary to the hypothesis.  Hence
\[
        1\le d\le L-2.
\]
The endpoint formula also identifies the two possible terminal vertices of $Q$ in this coordinate system:
\[
        \iota_P-a=\tau_P+dw=x_d,
        \qquad
        \iota_P-b=\tau_P+(d+1)w=x_{d+1}.
\]
The four endpoint cases are as follows:
\[
\begin{array}{c|c|c|c}
\iota_Q & \tau_Q & \gamma & \tau_Q \text{ in the }x_r\text{ notation}\\ \hline
\tau_P+a & \iota_P-a & 0 & x_d\\
\tau_P+b & \iota_P-b & 0 & x_{d+1}\\
\tau_P+b & \iota_P-a & w & x_d\\
\tau_P+a & \iota_P-b & -w & x_{d+1}
\end{array}
\]
The nonterminal cosets are already controlled by Lemma~\ref{lem:uniform-coset-shift}; it remains only to compare labels in the terminal coset.

\medskip
\noindent\emph{Case $\gamma=0$.}
Then $P\cap(Q+\gamma)=P\cap Q=\varnothing$.  The terminal inequalities follow from $1\le d\le L-2$ and the table above.

\medskip
\noindent\emph{Case $\gamma=w$.}
Here $\iota_Q=\tau_P+b$ and $\tau_Q=x_d$.  Applying the same terminal-coset formula to $Q$ gives $\iota_Q=\tau_Q+a+d_Qw$, hence
\[
        d_Qw=(\tau_P+b)-(\tau_P+dw)-a
        =b-a-dw=-(d+1)w,
\]
since $b-a=-w$.  Thus $d_Q\equiv L-d-1\pmod L$, and because $1\le d\le L-2$ we have $1\le L-d-1\le L-2$; hence $d_Q=L-d-1$.  In the terminal coordinate for $Q$, namely $y_r=\tau_Q+rw=x_d+rw$, the set $\{y_1,\ldots,y_{d_Q}\}$ is exactly $\{x_{d+1},\ldots,x_{L-1}\}$.  Thus
\[
        Q_b=\{x_{d+1},\ldots,x_{L-1}\},
        \qquad
        Q_a=\{x_0,\ldots,x_{d-1}\}.
\]
After translating by $w$,
\[
        (Q+w)_a=\{x_1,\ldots,x_d\},
        \qquad
        (Q+w)_b=\{x_0,x_{d+2},\ldots,x_{L-1}\}.
\]
Thus $(Q+w)_a\cap P_a=\varnothing$ and $(Q+w)_b\cap P_b=\varnothing$, because
\[
        P_a=\{x_{d+1},\ldots,x_{L-1}\},
        \qquad
        P_b=\{x_1,\ldots,x_d\}.
\]

\medskip
\noindent\emph{Case $\gamma=-w$.}
Here $\iota_Q=\tau_P+a$ and $\tau_Q=x_{d+1}$.  Again the endpoint formula for $Q$ gives
\[
        d_Qw=(\tau_P+a)-(\tau_P+(d+1)w)-a=-(d+1)w,
\]
so again $d_Q\equiv L-d-1\pmod L$.  The representative is $d_Q=L-d-1$ in the terminal coordinate $y_r=\tau_Q+rw=x_{d+1}+rw$; hence $\{y_1,\ldots,y_{d_Q}\}=\{x_{d+2},\ldots,x_{L-1},x_0\}$.  Therefore
\[
        Q_b=\{x_{d+2},\ldots,x_{L-1},x_0\},
        \qquad
        Q_a=\{x_1,\ldots,x_d\}.
\]
After translating by $-w$,
\[
        (Q-w)_a=\{x_0,x_1,\ldots,x_{d-1}\},
        \qquad
        (Q-w)_b=\{x_{d+1},\ldots,x_{L-1}\}.
\]
Thus $(Q-w)_a\cap P_a=\varnothing$ and $(Q-w)_b\cap P_b=\varnothing$.

The three cases give disjointness in the terminal coset, and the nonterminal cosets were handled above.  Therefore
\[
        P\cap(Q+\gamma)=\varnothing.
\]
Finally, since $1\le d\le L-2$, neither $x_d$ nor $x_{d+1}$ is $x_0=\tau_P$, and translating by $\gamma\in\{0,\pm w\}$ in the table never sends the listed terminal to $x_0$.  Hence
\[
        \tau_P\ne\tau_Q,
        \qquad
        \tau_P\ne\tau_Q+
        \gamma.
\]
Thus $(P,Q)$ is strongly switchable.
\end{proof}

\begin{proof}[Proof of Theorem~\ref{thm:threecycles}]
Let $D=\vec C_m\Boxprod\vec C_n$.  If $D$ has a Hamiltonian directed cycle, then $D\Boxprod\vec C_\ell$ contains the spanning subdigraph
\[
        \vec C_{mn}\Boxprod\vec C_\ell,
\]
and Theorem~\ref{input:twocycle-existence} gives two arc-disjoint Hamiltonian paths in that subdigraph.  These paths use only arcs of the subdigraph, hence they are also arc-disjoint Hamiltonian paths in the full product.

If $D$ has no Hamiltonian directed cycle, Lemma~\ref{lem:two-cycle-strong} gives a strongly switchable pair in $D$, and Lemma~\ref{lem:strong-lifting} lifts it to $D\Boxprod\vec C_\ell$.
\end{proof}

\begin{proof}[Proof of Corollary~\ref{cor:all-cycle-products}]
The case $r=2$ is \cite[Theorem~4.4]{DMM}, and the cases $r\ge4$ are \cite[Corollary~5.1]{DMM}.  The remaining case $r=3$ is Theorem~\ref{thm:threecycles}.
\end{proof}

\section*{Further directions}
A natural next problem is the two-generator nilpotent case.  Morris proved Hamiltonian paths for nilpotent Cayley digraphs; the packing problem asks which part of the cut-value reflection mechanism survives without commutativity.

The strongly switchable lifting criterion isolates the data needed to lift two disjoint Hamiltonian paths through one directed-cycle factor.  A higher-multiplicity analogue would have to control several relative translations and several terminal pairs at once; this is a basic obstruction for this method beyond two paths.

The cut-reflection theorem suggests a separate direction.  The proof depends only on the following part of the DMM data: an ordered primitive-ray system, the multiplicities $H_r=\lfloor N/L(R_r)\rfloor$, the cut values $U_r$ defined by the partial multiplicity sums, the endpoint identity $U_{f-1}+H_f=N$, and the sector filling lower bound.  This leads to a precise reflection problem for other lattice-parametrized standard-path families, especially when the arc-forcing index is larger than one and several cut-value sets occur.

\appendix
\section{Computational illustrations of cut reflection}\label{app:computations}

This appendix records small enumerations of the cut permutations $\Phi_d$, illustrating the cut-reflection theorem.  These computations are not used in the proof of Theorem~\ref{thm:main}.

\subsection{Computing the cut set}

For fixed $k$ and $a\not\equiv0,-1\pmod{k}$, define $\Phi_d$ for $0\le d<k$ by
\[
\Phi_d(i)=
\begin{cases}
 i+a+1,&0\le i<d,\\
 a,&i=d,\\
 i+a,&d<i\le k-1.
\end{cases}
\]
The Hamiltonian cut set and its reflected distance are
\[
        Z(k,a)=\{d\in\{0,\ldots,k-1\}:\Phi_d\text{ is a }k\text{-cycle}\},
        \qquad
        \Delta(k,a)=\min_{d,e\in Z(k,a)} |d-((k-1)-e)|.
\]
The set $Z(k,a)$ is obtained by one orbit computation for each permutation $\Phi_d$.  This is the finite enumeration of cut values behind Corollary~\ref{cor:cut-search}.

\subsection{Parity-sharp reflection in small examples}

Corollary~\ref{cor:parity-sharp} says that exact reflection occurs precisely when $k$ is odd, while even $k$ forces distance one.  The table below records representative cases.
\[
\begin{array}{c|c|c|c|c}
 k&a&Z&N-Z&\dist(Z,N-Z)\\ \hline
 5&2&\{0,4\}&\{0,4\}&0\\
 6&2&\{1,3\}&\{2,4\}&1\\
 10&4&\{1,3,5\}&\{4,6,8\}&1\\
 15&3&\{2,4,6,8,14\}&\{0,6,8,10,12\}&0
\end{array}
\]
For even $k$, the theorem gives a count pair with sum $k-2$ or $k$.  In the example $k=10,a=4$, both occur:
\[
        3+5=8=k-2,
        \qquad
        5+5=10=k.
\]
Both need not occur.  For instance, $k=10,a=3$ gives $Z=\{0,6,8\}$, so $k-2=8$ occurs but $k=10$ does not.  Conversely, $k=10,a=6$ gives $Z=\{1,3,9\}$, so $k=10$ occurs but $k-2=8$ does not.  Thus the sharp-modulo-parity theorem guarantees at least one central sum, but the distribution between $k-2$ and $k$ depends on finer arithmetic data.

\subsection{Questions suggested by the cut values}

The sharp-modulo-parity corollary completely explains when $Z\cap(N-Z)$ is nonempty in the first cyclic family.  Two further questions are suggested by the same computations.

First, when $k$ is even, characterize the parameters for which $k-2$, $k$, or both occur as sums of two Hamiltonian cut values.  Second, in the DMM parametrization with arc-forcing index greater than one, one obtains several cut-value sets rather than a single set $Z$.  A natural mixed version of the reflection problem is to ask whether suitable pairs of such sets satisfy estimates of the form
\[
        \dist(Z_t,N-Z_{t'})\le1.
\]
The proof of Theorem~\ref{thm:near-antipodal} suggests that any such extension would require analogues of two estimates: a reflected-mass upper bound and a sector-filling lower bound.

\section*{Disclosure of generative AI use}

GPT-5.5 Pro was used extensively in both the research and writing process for this paper.  In particular, it generated substantial portions of the draft text, proposed proof structures and case decompositions, supplied candidate arguments for many lemmas and propositions, suggested revisions after mathematical review, and assisted in checking notation, dependencies, exposition, and LaTeX source.  The author directed the project, selected the final arguments, verified the mathematical reasoning and computations, checked the cited literature to the extent used in the paper, and edited the final manuscript.  The author assumes full responsibility for the correctness, originality, and integrity of the work.  The AI system is not an author and is not cited as a source of mathematical authority.

\end{document}